\documentclass[12pt,reqno,twoside]{amsart}%

\usepackage{amsmath,amsthm,amscd,amsthm,upref,indentfirst}
\usepackage{amsfonts,mathrsfs}

\usepackage{amssymb,amsbsy,bm}
\usepackage{graphicx}
\usepackage{times}
\usepackage{amsaddr}
\usepackage{soul}
\usepackage{slashbox}

\usepackage{mathabx}

\usepackage[usenames]{color}

\theoremstyle{plain}
\newtheorem{theorem}{Theorem}

\newtheorem{proposition}[theorem]{Proposition}

\theoremstyle{definition}
\newtheorem{definition}[theorem]{Definition}

\newtheorem{corollary}[theorem]{Corollary}

\theoremstyle{remark}
\newtheorem{remark}[theorem]{Remark}

\newtheorem{example}[theorem]{Example}
\numberwithin{equation}{section}
\numberwithin{theorem}{section}

\usepackage[scaled=0.86]{helvet}
\renewcommand{\mathfrak}[1]{{\textbf{\upshape #1}}}
\renewcommand{\mathbf}{\bm}
\renewcommand{\mathrm}[1]{\scalebox{1.15}{\textsf{\upshape #1}}}
\renewcommand{\emph}[1]{\textrm{{\upshape #1}}}
\usepackage[scr=euler,scrscaled=1.15,bb=txof,bbscaled=1.16,cal=boondox,calscaled=1.15]{mathalfa}
\renewcommand{\mathit}[1]{\mathscr #1}
\renewcommand{\mathtt}[1]{\scalebox{1}{\bfseries \texttt{\upshape #1}}}
\usepackage{exscale}

\numberwithin{equation}{section}
\numberwithin{theorem}{section}

\usepackage[normalem]{ulem}
\usepackage[square,comma]{natbib}

\usepackage{mparhack}

\usepackage{keyval}
\usepackage{totcount}
\newtotcounter{citnum} 
\def\oldbibitem{} \let\oldbibitem=\bibitem
\def\bibitem{\stepcounter{citnum}\oldbibitem}
\newtotcounter{citdia} 
\def\oldxymatrix{} \let\oldxymatrix=\xymatrix
\def\xymatrix{\stepcounter{citdia}\oldxymatrix}

\usepackage[verbose]{backref}
\backrefsetup{verbose=false}
\renewcommand*{\backref}[1]{}
\renewcommand*{\backrefalt}[4]{[{\tiny%
    \ifcase #1 Not cited%
          \or Cited on page~\textcolor{BrickRed}{#2}%
          \else Cited on pages \textcolor{BrickRed}{#2}%
    \fi%
    }]}

\usepackage{fancyhdr}
\author{\small\scshape S\lowercase{teven} D\lowercase{uplij}}

\address{
Center for Information Technology (WWU IT),
Universit\"at M\"unster,
R\"ontgenstrasse 7-13\\
D-48149 M\"unster,
Deutschland}
\email{\small \sf douplii@uni-muenster.de;
sduplij@gmail.com;
https://ivv5hpp.uni-muenster.de/u/douplii}
\title{\large\bfseries\scshape M\lowercase{embership deformation of commutativity\\
 and obscure $n$-ary algebras}}

\date{\textit{of start} February 22, 2020. \textit{Date}:
\textit{of completion}
June 14, 2020.
\mbox{}\hskip 1.16em
\textit{Total}:
\total{citnum}
references
}

\renewcommand{\refname}{\textsc{References}}

\let\origsection\section
\renewcommand{\section}[1]{\sectionmark{#1}\origsection{#1}}
\let\origsubsection\subsection
\renewcommand{\subsection}[1]{\subsectionmark{#1}\origsubsection{#1}}

\makeatletter
\renewenvironment{thebibliography}[1]{%
  \@xp\origsection\@xp*\@xp{\refname}%
  \normalfont\footnotesize\labelsep .9em\relax
  \renewcommand\theenumiv{\arabic{enumiv}}\let\p@enumiv\@empty
  \vspace*{-5pt}
  \list{\@biblabel{\theenumiv}}{\settowidth\labelwidth{\@biblabel{#1}}%
    \leftmargin\labelwidth \advance\leftmargin\labelsep
    \usecounter{enumiv}}%
  \sloppy \clubpenalty\@M \widowpenalty\clubpenalty
  \sfcode`\.=\@m
}{%
  \def\@noitemerr{\@latex@warning{Empty `thebibliography' environment}}%
  \endlist
}
\makeatother

\subjclass[2010]{03E72, 08A72, 13A02, 16U80, 16W50, 17A42, 17A70, 17B05, 17B70, 17B75, 20C25, 20C35, 20F29, 20N15, 20N25, 94D05}
\keywords{associative algebra, almost commutative algebra, color algebra, obscure algebra, degree of truth, membership deformation, fuzzy set, vagueness, graded algebra, membership function, noncommutative algebra, n-ary algebra,  Lie algebra, projective representation, n-ary representation}

\begin{document}
\mbox{}
\maketitle
\mbox{}

\begin{abstract}

\noindent A general mechanism for \textquotedblleft breaking\textquotedblright%
\ commutativity in algebras is proposed: if the underlying set is taken to be
not a crisp set, but rather an obscure/fuzzy set, the membership function,
reflecting the degree of truth that an element belongs to the set, can be
incorporated into the commutation relations. The special \textquotedblleft
deformations\textquotedblright\ of commutativity and $\varepsilon
$-commutativity are introduced in such a way that equal degrees of truth
result in the \textquotedblleft nondeformed\textquotedblright\ case. We also
sketch how to \textquotedblleft deform\textquotedblright\ $\varepsilon$-Lie
algebras and Weyl algebras. Further, the above constructions are extended to
$n$-ary algebras for which the projective representations and $\varepsilon
$-commutativity are studied.

\end{abstract}

\thispagestyle{empty}

\mbox{}
\vspace{-0.5cm}

\begin{small}
\tableofcontents
\end{small}%

\newpage

\pagestyle{fancy}

\addtolength{\footskip}{15pt}

\renewcommand{\sectionmark}[1]{%
\markboth{
{ \scshape #1}}{}}

\renewcommand{\subsectionmark}[1]{%
\markright{
\mbox{\;}\\[5pt]
\textmd{#1}}{}}

\fancyhead{}
\fancyhead[EL,OR]{\leftmark}
\fancyhead[ER,OL]{\rightmark}
\fancyfoot[C]{\scshape -- \textcolor{BrickRed}{\thepage} --}
\renewcommand\headrulewidth{0.5pt}

\fancypagestyle{fancyref}{ %
\fancyhf{} 
\fancyhead[C]{\scshape R\lowercase{eferences} }
\fancyfoot[C]{\scshape -- \textcolor{BrickRed}{\thepage} --}
\renewcommand {\headrulewidth }{0.5pt}
\renewcommand {\footrulewidth }{0pt}
}

\fancypagestyle{emptyf}{
\fancyhead{}
\fancyfoot[C]{\scshape -- \textcolor{BrickRed}{\thepage} --}
\renewcommand{\headrulewidth}{0pt}
}

\mbox{}
\vskip 1.5cm

\thispagestyle{emptyf}

\section{\textsc{Introduction}}

Noncommutativity is the main mathematical idea of modern physics which, on the
one hand, is grounded in quantum mechanics, in which generators of the algebra
of observables do not commute, and, on the other hand, in supersymmetry and
its variations based on the graded commutativity concept. Informally,
\textquotedblleft deformation\textquotedblright\ of commutativity in algebras
is mostly a special way to place a \textquotedblleft scalar\textquotedblright%
\ multiplier from the algebra field before the permuted product of two
arbitrary elements. The general approach is based on projective representation
theory and realized using almost commutative ($\varepsilon$-commutative)
graded algebras, where the role of the multipliers is played by bicharacters
of the grading group (as suitable \textquotedblleft scalar\textquotedblright%
\ objects taking values in the algebra field $\Bbbk$).

Here we propose, principally, a new another mechanism for \textquotedblleft
deformation\textquotedblright\ of commutativity which comes from incorporating
the ideas of vagueness in logic \cite{smith} to algebra. First, take the
underlying set of the algebra not as a crisp set, but as an obscure/fuzzy set
\cite{dub/pra}. Second, consider an algebra (over $\Bbbk=\mathbb{C}$), such
that each element can be endowed with a membership function (representing the
degree of truth), a scalar function that takes values in the unit interval and
describes the containment\ of a given element in the underlying set
\cite{belohl}. Third, introduce a special \textquotedblleft membership
deformation\textquotedblright\ of the commutation relations, so to speak the
difference of the degree of truth, which determines a \textquotedblleft
measure of noncommutativity\textquotedblright, whereby the elements having
equal membership functions commute. Such procedure can be also interpreted as
the \textquotedblleft continuous noncommutativity\textquotedblright, because
the membership function is usually continuous. Likewise we \textquotedblleft
deform\textquotedblright\ $\varepsilon$-commutativity relations and
$\varepsilon$-Lie algebras \cite{sch79,rit/wyl}. Then we universalize and
apply the above constructions to $n$-ary algebras \cite{mic/vin} for which we
also study projective representations generalizing the binary ones
\cite{zmu72}.

\section{\textsc{Preliminaries}}

First recall the main features of the standard gradation concept and of
generalized (almost) commutativity (or $\varepsilon$-commutativity)
\cite{sch79,rit/wyl}.

Let $\Bbbk$ be a unital field (with unit $1\in\Bbbk$ and zero $0\in\Bbbk$) and
$\mathcal{A}=\left\langle A\mid\cdot,+\right\rangle $ be an associative
algebra over $\Bbbk$ having a zero $z\in A$ and unit $e\in A$ (for unital
algebras). A \textit{graded algebra} $\mathcal{A}_{\mathcal{G}}$ ($G$-graded
$\Bbbk$-algebra) is a direct sum of subalgebras $\mathcal{A}_{\mathcal{G}%
}=\bigoplus_{g\in G}\mathcal{A}_{g}$, where $\mathcal{G}=\left\langle
G\mid+^{\prime}\right\rangle $ is a \textit{grading group} (an abelian
(finite) group with \textquotedblleft unit\textquotedblright\ $\mathtt{0}\in
G$) and the set multiplication is (\textquotedblleft respecting
gradation\textquotedblright)
\begin{equation}
A_{g}\cdot A_{h}\subseteq A_{g+^{\prime}h},\ \ \ \ g,h\in G. \label{aaa}%
\end{equation}

Elements of subsets $a=a_{\left(  g\right)  }\in A_{g}$ (with
\textquotedblleft full membership\textquotedblright) are $G$%
-\textit{homogeneous} \textit{of degree} $g$
\begin{equation}
g\equiv\deg\left(  a_{\left(  g\right)  }\right)  =\deg\left(  a\right)
=a_{\left(  g\right)  }^{\prime}\equiv a^{\prime}\in G,\ \ \ a=a_{\left(
g\right)  }\in A_{g}. \label{i}%
\end{equation}

The graded algebra $\mathcal{A}_{\mathcal{G}}$ is called a \textit{cross
product}, if in each subalgebra $\mathcal{A}_{g}$ there exists at least one
invertible element. If all nonzero homogeneous elements are invertible, then
$\mathcal{A}_{\mathcal{G}}$ is called a \textit{graded division algebra}
\cite{dade80}. The morphisms $\varphi:\mathcal{A}_{\mathcal{G}}\rightarrow
\mathcal{B}_{\mathcal{G}}$ acting on homogeneous elements (from $A_{i}$)
should be compatible with the grading $\varphi\left(  A_{g}\right)  \subset
B_{g},\forall g\in G$, while $\ker\varphi$ is an homogeneous ideal. The
category of binary $G$-graded algebras $G$-$\mathtt{Alg}$ consists of the
corresponding class of algebras and the homogeneous morphisms (see, e.g.,
\cite{bourbaki98,dade80}).

In binary graded algebras there exists a way to generalize noncommutativity
such that it can be dependent on the gradings (\textquotedblleft
coloring\textquotedblright). Indeed, some (two-place) function on grading
degrees (bicharacter), a (\textit{binary})\textit{ commutation factor}
$\varepsilon^{\left(  2\right)  }:G\times G\rightarrow\Bbbk^{\times}$ (where
$\Bbbk^{\times}=\Bbbk\setminus0$) can be introduced \cite{sch79,rit/wyl} :
\begin{equation}
a\cdot b=\varepsilon^{\left(  2\right)  }\left(  a^{\prime},b^{\prime}\right)
b\cdot a,\ \ \ \forall a,b\in A,\ \ \forall a^{\prime},b^{\prime}\in
G.\label{ab}%
\end{equation}

The properties of the commutation factor $\varepsilon^{\left(  2\right)  }$
under double permutation and associativity%
\begin{align}
\varepsilon^{\left(  2\right)  }\left(  a^{\prime},b^{\prime}\right)
\varepsilon^{\left(  2\right)  }\left(  b^{\prime},a^{\prime}\right)   &
=1,\label{e1}\\
\varepsilon^{\left(  2\right)  }\left(  a^{\prime},b^{\prime}+c^{\prime
}\right)   &  =\varepsilon^{\left(  2\right)  }\left(  a^{\prime},b^{\prime
}\right)  \varepsilon^{\left(  2\right)  }\left(  a^{\prime},c^{\prime
}\right)  ,\label{e2}\\
\varepsilon^{\left(  2\right)  }\left(  a^{\prime}+b^{\prime},c^{\prime
}\right)   &  =\varepsilon^{\left(  2\right)  }\left(  a^{\prime},c^{\prime
}\right)  \varepsilon^{\left(  2\right)  }\left(  b^{\prime},c^{\prime
}\right)  ,\ \ \forall a^{\prime},b^{\prime},c^{\prime}\in G, \label{e3}%
\end{align}
make $\varepsilon^{\left(  2\right)  }$ a special 2-cocycle on the group
$\mathcal{G}$ \cite{bourbaki98}. The conditions (\ref{e1})--(\ref{e3}) imply
that $\varepsilon^{\left(  2\right)  }\left(  a^{\prime},b^{\prime}\right)
\neq0$, $\left(  \varepsilon^{\left(  2\right)  }\left(  a^{\prime},a^{\prime
}\right)  \right)  ^{2}=1$, $\varepsilon^{\left(  2\right)  }\left(
a^{\prime},\mathtt{0}\right)  =\varepsilon^{\left(  2\right)  }\left(
\mathtt{0},a^{\prime}\right)  =1$, and $\mathtt{0}\in G$. A graded algebra
$\mathcal{A}_{\mathcal{G}}$ endowed with the commutation factor $\varepsilon
^{\left(  2\right)  }$ satisfying (\ref{ab})--(\ref{e3}) is called an
\textit{almost commutative }($\varepsilon^{\left(  2\right)  }$%
-\textit{commutative}, \textit{color}) \textit{algebra} \cite{sch79,rit/wyl}
(for a review, see \cite{gou/mas/wal}).

The simplest example of a commutation factor is a \textit{sign rule}%
\begin{equation}
\varepsilon^{\left(  2\right)  }\left(  a^{\prime},b^{\prime}\right)  =\left(
-1\right)  ^{<a^{\prime},b^{\prime}>}, \label{e}%
\end{equation}
where $<\ ,\ >:G\times G\rightarrow\mathbb{Z}_{2}$ is a bilinear form
(\textquotedblleft scalar product\textquotedblright), and for $G=\mathbb{Z}%
_{2}$ the form is a product i.e. $<a^{\prime},b^{\prime}>=a^{\prime}b^{\prime
}\equiv gh\in\mathbb{Z}_{2}$. This gives the standard supercommutative algebra
\cite{kac3,nah/rit/sch,berezin}.

In the case $G=\mathbb{Z}_{2}^{n}$ the \textquotedblleft scalar
product\textquotedblright\ $<\ ,\ >:G\times G\rightarrow\mathbb{Z}$ is defined
by (see (\ref{i}))%
\begin{equation}
<\left(  a_{1}^{\prime}\ldots a_{n}^{\prime}\right)  ,\left(  b_{1}^{\prime
}\ldots b_{n}^{\prime}\right)  >=a_{1}^{\prime}b_{1}^{\prime}+\ldots
+a_{n}^{\prime}b_{n}^{\prime}\equiv g_{1}h_{1}+\ldots+g_{n}h_{n}\in\mathbb{Z},
\end{equation}
which leads to $\mathbb{Z}_{2}^{n}$-commutative associative algebras
\cite{Cov/Gra/Pon1}.

A classification of the commutation factors $\varepsilon^{\left(  2\right)  }$
can be made in terms of the \textit{factors} (binary \textit{Schur
multipliers}) $\pi^{\left(  2\right)  }:G\times G\rightarrow\Bbbk^{\times}$
such that \cite{zmu60,sch79}%
\begin{equation}
\varepsilon_{\pi}^{\left(  2\right)  }\left(  a^{\prime},b^{\prime}\right)
=\frac{\pi^{\left(  2\right)  }\left(  a^{\prime},b^{\prime}\right)  }%
{\pi^{\left(  2\right)  }\left(  b^{\prime},a^{\prime}\right)  },\ \ \forall
a^{\prime},b^{\prime}\in G. \label{es}%
\end{equation}

The (Schur) factors $\pi^{\left(  2\right)  }$ naturally appear in projective
representation theory \cite{zmu72} and satisfy the relation%
\begin{equation}
\pi^{\left(  2\right)  }\left(  a^{\prime},b^{\prime}+c^{\prime}\right)
\pi^{\left(  2\right)  }\left(  b^{\prime},c^{\prime}\right)  =\pi^{\left(
2\right)  }\left(  a^{\prime},b^{\prime}\right)  \pi^{\left(  2\right)
}\left(  a^{\prime}+b^{\prime},c^{\prime}\right)  ,\ \ \forall a^{\prime
},b^{\prime}\in G, \label{sa}%
\end{equation}
which follows from associativity. Using (\ref{es}) we can rewrite the
$\varepsilon^{\left(  2\right)  }$-commutation relation (\ref{ab}) of the
graded algebra $\mathcal{A}_{\mathcal{G}}$ in the form%
\begin{equation}
\pi^{\left(  2\right)  }\left(  b^{\prime},a^{\prime}\right)  a\cdot
b=\pi^{\left(  2\right)  }\left(  a^{\prime},b^{\prime}\right)  b\cdot
a,\ \ \ \forall a,b\in A,\ \ \forall a^{\prime},b^{\prime}\in G. \label{sab}%
\end{equation}
We call (\ref{sab}) a $\pi$\textit{-commutativity}. A factor set $\left\{
\pi_{sym}^{\left(  2\right)  }\right\}  $ is \textit{symmetric}, if%
\begin{align}
\pi_{sym}^{\left(  2\right)  }\left(  b^{\prime},a^{\prime}\right)   &
=\pi_{sym}^{\left(  2\right)  }\left(  a^{\prime},b^{\prime}\right)
,\label{psym}\\
\varepsilon_{\pi_{sym}}^{\left(  2\right)  }\left(  a^{\prime},b^{\prime
}\right)   &  =1,\ \ \ \forall a^{\prime},b^{\prime}\in G, \label{esym}%
\end{align}
and so the graded algebra $\mathcal{A}_{\mathcal{G}}$ becomes commutative.

For further details, see \textsc{Section} \ref{sec-proj} and
\cite{sch79,rit/wyl}.

\section{\label{sec-memb}\textsc{Membership function and obscure algebras}}

Now let us consider a generalization of associative algebras and graded
algebras to the case where the degree of truth (containment\ of an element in
the underlying set) is not full or distinct (for a review, see, e.g.
\cite{dub/pra,zim11}). In this case, an element $a\in A$ of the algebra
$\mathcal{A}=\left\langle A\mid\cdot,+\right\rangle $ (over $\Bbbk$) can be
endowed with an additional function, the \textit{membership function}
$\mu:A\rightarrow\left[  0,1\right]  $, $0,1\in\Bbbk$, which \textquotedblleft
measures\textquotedblright\ the degree of truth as a \textquotedblleft grade
of membership\textquotedblright\ \cite{belohl}. If $A$ is a \textit{crisp set}
then $\mu\in\left\{  0,1\right\}  $ and $\mu^{crisp}\left(  a\right)  =1$, if
$a\in A$ and $\mu^{crisp}\left(  a\right)  =0$, if $a\notin A$. Also, we
assume that the zero has full membership $\mu\left(  z\right)  =1$, $z\in A$
(for details, see, e.g. \cite{dub/pra,zim11}). If the membership function
$\mu$ is positive, we can identify an \textit{obscure }(\textit{fuzzy}%
)\textit{ set} $\mathfrak{A}^{\left(  \mu\right)  }$ with support the
universal set $A$ consisting of pairs%
\begin{equation}
\mathfrak{A}^{\left(  \mu\right)  }=\left\{  \left(  a|\mu\left(  a\right)
\right)  \right\}  ,a\in A,\mu\left(  a\right)  >0.\label{am}%
\end{equation}

Sometimes, instead of operations with obscure sets it is convenient to
consider the corresponding operations only in terms of the membership function
$\mu$ itself. Denote $a\wedge b=\min\left\{  a,b\right\}  $, $a\vee
b=\max\left\{  a,b\right\}  $. Then for inclusion $\mathbf{\subseteq}$, union
$\mathbf{\cup}$, intersection $\mathbf{\cap}$ and negation of obscure sets one
can write $\mu\left(  a\right)  \leq\mu\left(  b\right)  $, $\mu\left(
a\right)  \vee\mu\left(  b\right)  $, $\mu\left(  a\right)  \wedge\mu\left(
b\right)  $, $1-\mu\left(  a\right)  $, $a,b\in$ $\mathfrak{A}^{\left(
\mu\right)  }$, respectively.

\begin{definition}
\label{def-balg}An \textit{obscure algebra }$\mathcal{A}\left(  \mu\right)
=\left\langle \mathfrak{A}^{\left(  \mu\right)  }\mid\cdot,+\right\rangle $ is
an algebra over $\Bbbk$ having an obscure set $\mathfrak{A}^{\left(
\mu\right)  }$ as its underlying set, where the following conditions hold%
\begin{align}
\mu\left(  a+b\right)   &  \geq\mu\left(  a\right)  \wedge\mu\left(  b\right)
,\label{m1}\\
\mu\left(  a\cdot b\right)   &  \geq\mu\left(  a\right)  \wedge\mu\left(
b\right)  ,\label{m2}\\
\mu\left(  ka\right)   &  \geq\mu\left(  a\right)  ,\ \ \forall a,b\in
A,\ \ k\in\Bbbk. \label{m3}%
\end{align}

\end{definition}

\begin{definition}
A \textit{obscure unity} $\eta$ is given by $\eta\left(  a\right)  =1$, if
$a=0$ and $\eta\left(  a\right)  =0$, if $a\neq0$.
\end{definition}

For two obscure sets their direct sum $\mathfrak{A}^{\left(  \mu\right)
}=\mathfrak{A}^{\left(  \mu_{g}\right)  }\mathbf{\oplus}$ $\mathfrak{A}%
^{\left(  \mu_{h}\right)  }$ can be defined, if $\mu_{g}\wedge\mu_{h}=\eta$,
and the joint membership function has two \textquotedblleft
nonintersecting\textquotedblright\ components%
\begin{align}
\mu\left(  a\right)   &  =\mu_{g}\left(  a_{\left(  g\right)  }\right)
\vee\mu_{h}\left(  a_{\left(  h\right)  }\right)  ,\mu_{g}\wedge\mu_{h}%
=\eta,\\
a &  =a_{\left(  g\right)  }+a_{\left(  h\right)  },\ a\in\mathfrak{A}%
^{\left(  \mu\right)  },\ a_{\left(  g\right)  }\in\mathfrak{A}^{\left(
\mu_{g}\right)  },\ a_{\left(  h\right)  }\in\mathfrak{A}^{\left(  \mu
_{h}\right)  },\ g,h\in G,
\end{align}
satisfying%
\begin{equation}
\mu\left(  a_{\left(  g\right)  }\cdot a_{\left(  h\right)  }\right)  \geq
\mu_{g}\left(  a_{\left(  g\right)  }\right)  \wedge\mu_{h}\left(  a_{\left(
h\right)  }\right)  .
\end{equation}

\begin{definition}
An\textit{ obscure} $G$-\textit{graded algebra} is a direct sum decomposition%
\begin{equation}
\mathcal{A}_{\mathcal{G}}\left(  \mu\right)  =\mathbf{\bigoplus}_{g\in
G}\mathcal{A}\left(  \mu_{g}\right)  , \label{uu}%
\end{equation}
such that the relation (\ref{aaa}) holds and the joint membership function is%
\begin{equation}
\mu\left(  a\right)  =\underset{g\in G}{\mathbf{\vee}}\mu_{g}\left(
a_{\left(  g\right)  }\right)  ,\ \ \ \ \ a=\sum_{g\in G}a_{\left(  g\right)
},\ \forall a\in\mathfrak{A}^{\left(  \mu\right)  },\ a_{\left(  g\right)
}\in\mathfrak{A}^{\left(  \mu_{g}\right)  },\ g\in G. \label{ma}%
\end{equation}

\end{definition}

\section{\textsc{Membership deformation of commutativity}}

The membership concept leads to the question whether is it possible to
generalize commutativity and $\varepsilon^{\left(  2\right)  }$-commutativity
(\ref{ab}) for the obscure algebras? A positive answer to this question can be
based on a consistent usage of the membership $\mu$ as ordinary functions
which are pre-defined for each element and satisfy some conditions (see, e.g.
\cite{dub/pra,zim11}). In this case, the commutation factor (and the Schur
factors) may depend not only on the element grading, but also on the element
membership function $\mu$, and therefore it becomes \textquotedblleft
individual\textquotedblright\ for each pair of elements, and moreover it can
be continuous. We call this procedure a \textit{membership deformation} of
commutativity\textit{.}

\subsection{Deformation of commutative algebras}

Let us consider an obscure commutative algebra $\mathcal{A}\left(  \mu\right)
=\left\langle \mathfrak{A}^{\left(  \mu\right)  }\mid\cdot,+\right\rangle $
(see \textbf{Definition} \ref{def-balg}), in which $a\cdot b=b\cdot a$,
$\forall a,b\in\mathfrak{A}^{\left(  \mu\right)  }$. Now we \textquotedblleft
deform\textquotedblright\ this commutativity by the membership function $\mu$
and introduce a new algebra product $\left(  \ast\right)  $ for the elements
in $A$ to get a noncommutative algebra.

\begin{definition}
An \textit{obscure membership deformed algebra} is $\mathcal{A}_{\ast}\left(
\mu\right)  =\left\langle \mathfrak{A}^{\left(  \mu\right)  }\mid
\ast,+\right\rangle $ in which the noncommutativity relation is%
\begin{equation}
\mu\left(  b\right)  a\ast b=\mu\left(  a\right)  b\ast a,\ \ \ \forall
a,b\in\mathfrak{A}^{\left(  \mu\right)  }.\label{mab}%
\end{equation}

\end{definition}

\begin{remark}
The relation (\ref{mab}) is reminiscent of (\ref{sab}), where the role of the
Schur factors is played by the membership function $\mu$ which depends on the
element itself, but not on the element's grading. Therefore, the membership
deformation of commutativity (\ref{mab}) is highly nonlinear as opposed to gradation.
\end{remark}

Because in our consideration the membership function $\mu$ cannot be zero, it
follows from (\ref{mab}) that%
\begin{align}
a\ast b  &  =\mathbf{\epsilon}_{\mu}^{\left(  2\right)  }\left(  a,b\right)
b\ast a,\label{abe}\\
\mathbf{\epsilon}_{\mu}^{\left(  2\right)  }\left(  a,b\right)   &  =\frac
{\mu\left(  a\right)  }{\mu\left(  b\right)  },\ \ \ \forall a,b\in
\mathfrak{A}^{\left(  \mu\right)  }, \label{emm}%
\end{align}
where $\mathbf{\epsilon}_{\mu}^{\left(  2\right)  }$ is the \textit{membership
commutation factor}.

\begin{corollary}
An obscure membership deformed algebra $\mathcal{A}_{\ast}\left(  \mu\right)
$ is a (kind of) $\mathbf{\epsilon}_{\mu}$-commutative algebra with a
membership commutation factor (\ref{emm}) which now depends not on the element
gradings (as in (\ref{es})), but on the membership function $\mu$.
\end{corollary}

\begin{remark}
As can be seen from (\ref{emm}), the noncommutativity \textquotedblleft
measures\textquotedblright\ the difference between the element degree of
truth\ and $\mathfrak{A}^{\left(  \mu\right)  }$. So if two elements have the
same membership function, they commute.
\end{remark}

\begin{proposition}
In $\mathcal{A}_{\ast}^{\left(  \mu\right)  }$ the membership function
satisfies the equality (cf. (\ref{m3}))%
\begin{equation}
\mu\left(  ka\right)  =\mu\left(  a\right)  ,\ \ \ \forall a\in\mathfrak{A}%
^{\left(  \mu\right)  },\ \ k\in\Bbbk. \label{mka}%
\end{equation}

\end{proposition}

\begin{proof}
From the distributivity of the scalar multiplication $k\left(  a\ast b\right)
=ka\ast b=a\ast kb$ and the membership noncommutativity (\ref{mab}), we get%
\begin{align}
\mu\left(  b\right)  ka\ast b &  =\mu\left(  ka\right)  b\ast ka,\label{mk1}\\
\mu\left(  kb\right)  a\ast kb &  =\mu\left(  a\right)  kb\ast a,\label{mk2}%
\end{align}
and so $\mu\left(  ka\right)  \mu\left(  kb\right)  =\mu\left(  a\right)
\mu\left(  b\right)  $, $\forall a\in\mathfrak{A}^{\left(  \mu\right)
},\ \ k\in\Bbbk$, and thus (\ref{mka}) follows.
\end{proof}

In general, the algebra $\mathcal{A}_{\ast}\left(  \mu\right)  $ is not
associative without further conditions (for instance, similar to (\ref{sa}))
on the membership function $\mu$ which is assumed predefined and satisfies the
properties (\ref{m1})--(\ref{m2}) and (\ref{mka}) only.

\begin{proposition}
\label{asser-assoc}The obscure membership deformed algebra $\mathcal{A}_{\ast
}\left(  \mu\right)  $ cannot be associative with any additional conditions on
the membership function $\mu$.
\end{proposition}

\begin{proof}
Since the $\varepsilon$-commutativity (\ref{ab}) and the $\mathbf{\epsilon}%
$-commutativity (\ref{abe}) have the same form, the derivations of
associativity coincide and give a cocycle relation similar to (\ref{e2}%
)--(\ref{e3}) also for $\mathbf{\epsilon}_{\mu}^{\left(  2\right)  }$, e.g.%
\begin{equation}
\mathbf{\epsilon}_{\mu}^{\left(  2\right)  }\left(  a,b\ast c\right)
=\mathbf{\epsilon}_{\mu}^{\left(  2\right)  }\left(  a,b\right)
\mathbf{\epsilon}_{\mu}^{\left(  2\right)  }\left(  a,c\right)  .
\end{equation}
In terms of the membership function (\ref{emm}) this becomes $\mu\left(
a\right)  =\mu\left(  b\right)  \mu\left(  c\right)  \diagup\mu\left(  b\ast
c\right)  $, but this is impossible for all $a,b,c\in\mathfrak{A}^{\left(
\mu\right)  }$. On the other side, after double commutation in $\left(  a\ast
b\right)  \ast c\rightarrow c\ast\left(  b\ast a\right)  $ and $a\ast\left(
b\ast c\right)  \rightarrow\left(  c\ast b\right)  \ast a$ we obtain (if the
associativity of $\left(  \ast\right)  $ is implied)%
\begin{equation}
\mathbf{\epsilon}_{\mu}^{\left(  2\right)  }\left(  a,b\ast c\right)
\mathbf{\epsilon}_{\mu}^{\left(  2\right)  }\left(  b,c\right)
=\mathbf{\epsilon}_{\mu}^{\left(  2\right)  }\left(  a\ast b,c\right)
\mathbf{\epsilon}_{\mu}^{\left(  2\right)  }\left(  a,b\right)  ,
\end{equation}
which in terms of the membership function becomes $\mu\left(  b\right)
^{2}=\mu\left(  a\ast b\right)  \mu\left(  b\ast c\right)  $, and this is also
impossible for arbitrary $a,b,c\in\mathfrak{A}^{\left(  \mu\right)  }$.
\end{proof}

Nevertheless, distributivity of the algebra multiplication and algebra
addition in $\mathcal{A}_{\ast}\left(  \mu\right)  $ is possible, but can only
be one-sided.

\begin{proposition}
The algebra $\mathcal{A}_{\ast}\left(  \mu\right)  $ is right-distributive,
but has the membership deformed left distributivity%
\begin{align}
\mu\left(  b+c\right)  a\ast\left(  b+c\right)   &  =\mu\left(  b\right)
a\ast b+\mu\left(  c\right)  a\ast c,\label{dl}\\
\left(  b+c\right)  \ast a &  =b\ast a+c\ast a,\ \ \ \forall a,b,c\in
\mathfrak{A}^{\left(  \mu\right)  }.\label{dr}%
\end{align}

\end{proposition}

\begin{proof}
Applying the membership noncommutativity (\ref{mab}) to (\ref{dl}), we obtain
$\mu\left(  a\right)  \left(  b+c\right)  \ast a=\mu\left(  a\right)  b\ast
a+\mu\left(  a\right)  c\ast a$, and then (\ref{dr}), because $\mu>0$.
\end{proof}

\begin{theorem}
The binary obscure algebra $\mathcal{A}_{\ast}\left(  \mu\right)
=\left\langle \mathfrak{A}^{\left(  \mu\right)  }\mid\ast,+\right\rangle $ is
necessarily nonassociative and $\mathbf{\epsilon}_{\mu}^{\left(  2\right)  }%
$-commutative (\ref{abe}), right-distributive (\ref{dr}) and membership
deformed left-distributive (\ref{dl}).
\end{theorem}

\begin{remark}
If membership noncommutativity is valid for generators of the algebra only,
then the form of (\ref{abe}) coincides with that of the quantum polynomial
algebra \cite{art97}, but the latter is two-sided distributive, in distinction
to the obscure algebra $\mathcal{A}_{\ast}\left(  \mu\right)  $.
\end{remark}

\begin{example}
[\textsf{Deformed Weyl algebra}]Consider an obscure algebra $\mathcal{A}%
_{\odot}^{Weyl}\left(  \mu\right)  =\left\langle \mathfrak{A}^{\left(
\mu\right)  }\mid\odot,+\right\rangle $ generated by two generators
$x,y\in\mathfrak{A}^{\left(  \mu\right)  }$ satisfying the Weyl-like relation
(cf. (\ref{mab}))%
\begin{equation}
\mu\left(  y\right)  x\odot y=\mu\left(  x\right)  y\odot x+\mu\left(  x\odot
y\right)  e,\ \ \ x,y,e\in\mathfrak{A}^{\left(  \mu\right)  }. \label{mxy}%
\end{equation}
We call $\mathcal{A}_{\odot}^{Weyl}\left(  \mu\right)  $ a \textit{membership
deformed Weyl algebra}. Because the membership function is predefined, and
$\mu\left(  x\odot y\right)  $ is not a symplectic form, the algebra
$\mathcal{A}_{\odot}^{Weyl}\left(  \mu\right)  $ is not isomorphic to the
ordinary Weyl algebra (see, e.g., \cite{lam01}). In the same way, the graded
Weyl algebra \cite{til77} can be membership deformed in a similar way.
\end{example}

\begin{remark}
The special property of the membership\ noncommutativity is the fact that each
pair of elements has their own \textquotedblleft individual\textquotedblright%
\ commutation factor $\mathbf{\epsilon}$ which depends on the membership
function (\ref{emm}) that can also be continuous.
\end{remark}

\subsection{\label{subsec-def-e}Deformation of $\varepsilon$-commutative
algebras}

Here we apply the membership deformation procedure (\ref{mab}) to the obscure
$G$-graded algebras (\ref{uu}) which are $\varepsilon^{\left(  2\right)  }%
$-commutative (\ref{ab}). We now \textquotedblleft deform\textquotedblright%
\ (\ref{sab}) by analogy with (\ref{mab}).

Let $\mathcal{A}_{\mathcal{G}}\left(  \mu\right)  $ be a binary obscure
$G$-graded algebra (\ref{uu}) which is $\varepsilon^{\left(  2\right)  }%
$-commutative with the Schur factor $\pi^{\left(  2\right)  }$ (\ref{es}).

\begin{definition}
An \textit{obscure membership deformed }$\varepsilon^{\left(  2\right)  }%
$\textit{-commutative }$G$\textit{-graded algebra} is $\mathcal{A}%
_{\mathcal{G}\star}\left(  \mu\right)  =\left\langle \mathfrak{A}^{\left(
\mu\right)  }\mid\star,+\right\rangle $ in which the noncommutativity relation
is given by%
\begin{equation}
\mu\left(  b\right)  \pi^{\left(  2\right)  }\left(  b^{\prime},a^{\prime
}\right)  a\star b=\mu\left(  a\right)  \pi^{\left(  2\right)  }\left(
a^{\prime},b^{\prime}\right)  b\star a,\ \ \ \forall a,b\in\mathfrak{A}%
^{\left(  \mu\right)  },\forall a^{\prime},b^{\prime}\in G.\label{mp2}%
\end{equation}

\end{definition}

Because both functions $\pi^{\left(  2\right)  }$ and $\mu$ are nonvanishing,
we can combine (\ref{es}) and (\ref{emm}).

\begin{definition}
\label{def-doub-ee} An algebra $\mathcal{A}_{\mathcal{G}\star}\left(
\mu\right)  $ is called a \textit{double} $\varepsilon_{\pi}^{\left(
2\right)  }/\mathbf{\epsilon}_{\mu}^{\left(  2\right)  }$-\textit{commutative
algebra}, when%
\begin{align}
a\star b  &  =\varepsilon_{\pi}^{\left(  2\right)  }\left(  a^{\prime
},b^{\prime}\right)  \mathbf{\epsilon}_{\mu}^{\left(  2\right)  }\left(
a,b\right)  b\star a,\ \ \ \forall a,b\in\mathfrak{A}^{\left(  \mu\right)
},\forall a^{\prime},b^{\prime}\in G\label{ep}\\
\ \ \mathbf{\epsilon}_{\mu}^{\left(  2\right)  }\left(  a,b\right)   &
=\frac{\mu\left(  a\right)  }{\mu\left(  b\right)  },\label{em}\\
\varepsilon_{\pi}^{\left(  2\right)  }\left(  a^{\prime},b^{\prime}\right)
&  =\frac{\pi^{\left(  2\right)  }\left(  a^{\prime},b^{\prime}\right)  }%
{\pi^{\left(  2\right)  }\left(  b^{\prime},a^{\prime}\right)  },
\end{align}
where $\varepsilon_{\pi}^{\left(  2\right)  }$ is the grading commutation
factor and $\mathbf{\epsilon}_{\mu}^{\left(  2\right)  }$ is the membership
commutation factor.
\end{definition}

In the \textsf{first version}, we assume that $\varepsilon_{\pi}^{\left(
2\right)  }$ is still a cocycle and satisfies (\ref{e1})--(\ref{e3}), then in
$\mathcal{A}_{\mathcal{G}\star}\left(  \mu\right)  $ the relation (\ref{mka})
is satisfied as well, because the Schur factors cancel in the derivation from
(\ref{mk1})--(\ref{mk2}). For the same reason \textbf{Assertion}
\ref{asser-assoc} holds, and therefore the double $\varepsilon_{\pi}^{\left(
2\right)  }/\mathbf{\epsilon}_{\mu}^{\left(  2\right)  }$-commutative algebra
$\mathcal{A}_{\mathcal{G}\star}\left(  \mu\right)  $ with the fixed grading
commutation factor $\varepsilon_{\pi}^{\left(  2\right)  }$ is necessarily nonassociative.

As the \textsf{second version}, we consider the case when the grading
commutation factor does not satisfy (\ref{e1})--(\ref{e3}), but the double
commutation factor does satisfy them (the membership function is fixed, being
predefined for each element), which can lead to an associative algebra.

\begin{proposition}
A double $\varepsilon_{\pi\mu}^{\left(  2\right)  }/\mathbf{\epsilon}_{\mu
}^{\left(  2\right)  }$-commutative algebra $\mathcal{A}_{\mathcal{G}%
\circledast}\left(  \mu\right)  =\left\langle \mathfrak{A}^{\left(
\mu\right)  }\mid\circledast,+\right\rangle $ with%
\begin{equation}
a\circledast b=\varepsilon_{\pi\mu}^{\left(  2\right)  }\left(  a^{\prime
},b^{\prime}\right)  \mathbf{\epsilon}_{\mu}^{\left(  2\right)  }\left(
a,b\right)  b\circledast a,\ \ \mathbf{\epsilon}_{\mu}^{\left(  2\right)
}\left(  a,b\right)  =\frac{\mu\left(  a\right)  }{\mu\left(  b\right)
}\label{ab1}%
\end{equation}
is associative, if the noncocycle commutation factor $\varepsilon_{\pi\mu
}^{\left(  2\right)  }$ satisfies the \textquotedblleft membership deformed
cocycle-like\textquotedblright\ conditions%
\begin{align}
\varepsilon_{\pi\mu}^{\left(  2\right)  }\left(  a^{\prime},b^{\prime}\right)
\varepsilon_{\pi\mu}^{\left(  2\right)  }\left(  b^{\prime},a^{\prime}\right)
&  =\frac{1}{\mathbf{\epsilon}_{\mu}^{\left(  2\right)  }\left(  a,b\right)
\mathbf{\epsilon}_{\mu}^{\left(  2\right)  }\left(  b,a\right)  }%
,\label{em1}\\
\varepsilon_{\pi\mu}^{\left(  2\right)  }\left(  a^{\prime},b^{\prime
}+c^{\prime}\right)   &  =\varepsilon_{\pi\mu}^{\left(  2\right)  }\left(
a^{\prime},b^{\prime}\right)  \varepsilon_{\pi\mu}^{\left(  2\right)  }\left(
a^{\prime},c^{\prime}\right)  \frac{\mathbf{\epsilon}_{\mu}^{\left(  2\right)
}\left(  a,b\right)  \mathbf{\epsilon}_{\mu}^{\left(  2\right)  }\left(
a,c\right)  }{\mathbf{\epsilon}_{\mu}^{\left(  2\right)  }\left(
a,b\circledast c\right)  },\label{em2}\\
\varepsilon_{\pi\mu}^{\left(  2\right)  }\left(  a^{\prime}+b^{\prime
},c^{\prime}\right)   &  =\varepsilon_{\pi\mu}^{\left(  2\right)  }\left(
a^{\prime},c^{\prime}\right)  \varepsilon_{\pi\mu}^{\left(  2\right)  }\left(
b^{\prime},c^{\prime}\right)  \frac{\mathbf{\epsilon}_{\mu}^{\left(  2\right)
}\left(  a,c\right)  \mathbf{\epsilon}_{\mu}^{\left(  2\right)  }\left(
b,c\right)  }{\mathbf{\epsilon}_{\mu}^{\left(  2\right)  }\left(  a\circledast
b,c\right)  },\label{em3}\\
\forall a,b &  \in\mathfrak{A}^{\left(  \mu\right)  },\forall a^{\prime
},b^{\prime}\in G.\nonumber
\end{align}

\begin{proof}
Now indeed the double commutation factor (the product of the grading and
membership factors) $\varepsilon_{\pi\mu}^{\left(  2\right)  }\mathbf{\epsilon
}_{\mu}^{\left(  2\right)  }$ satisfies (\ref{e1})--(\ref{e3}). Then
(\ref{em1})--(\ref{em3}) immediately follow.
\end{proof}
\end{proposition}

We can find the deformed equation for the Schur-like factors (similar to
(\ref{em1})--(\ref{em3}))%
\begin{equation}
\varepsilon_{\pi\mu}^{\left(  2\right)  }\left(  a^{\prime},b^{\prime}\right)
=\frac{\pi_{\mu}^{\left(  2\right)  }\left(  a^{\prime},b^{\prime}\right)
}{\pi_{\mu}^{\left(  2\right)  }\left(  b^{\prime},a^{\prime}\right)  },
\label{epm}%
\end{equation}
such that the following \textquotedblleft membership
deformed\textquotedblright\ commutation takes place (see (\ref{ab1}))%
\begin{equation}
\pi_{\mu}^{\left(  2\right)  }\left(  b^{\prime},a^{\prime}\right)  \mu\left(
b\right)  a\circledast b=\pi_{\mu}^{\left(  2\right)  }\left(  a^{\prime
},b^{\prime}\right)  \mu\left(  a\right)  b\circledast a,\ \ \ \forall
a,b\in\mathfrak{A}^{\left(  \mu\right)  },\forall a^{\prime},b^{\prime}\in G,
\label{pm}%
\end{equation}
and the algebra $\mathcal{A}_{\mathcal{G}\circledast}\left(  \mu\right)  $
becomes associative (in distinct with (\ref{mp2}), where the grading
commutation factor $\varepsilon_{\pi}^{\left(  2\right)  }$, satisfies
(\ref{e1})--(\ref{e3}), and the algebra multiplications $\left(  \star\right)
$ are different).

Thus the deformed binary Schur-like factors $\pi_{\mu}^{\left(  2\right)  }$
of the obscure membership deformed associative\textit{ }double commutative
algebra $\mathcal{A}_{\mathcal{G}\circledast}\left(  \mu\right)  $ satisfy%
\begin{align}
\pi_{\mu}^{\left(  2\right)  }\left(  a^{\prime},b^{\prime}+c^{\prime}\right)
\pi_{\mu}^{\left(  2\right)  }\left(  b^{\prime},c^{\prime}\right)   &
=\pi_{\mu}^{\left(  2\right)  }\left(  a^{\prime},b^{\prime}\right)  \pi_{\mu
}^{\left(  2\right)  }\left(  a^{\prime}+b^{\prime},c^{\prime}\right)
\frac{\mu\left(  a\circledast b\right)  }{\mu\left(  b\right)  },\\
\forall a,b &  \in\mathfrak{A}^{\left(  \mu\right)  },\ \ \forall a^{\prime
},b^{\prime}\in G,\nonumber
\end{align}
which should be compared with the corresponding nondeformed relation (\ref{sa}).

\subsection{Double $\varepsilon\epsilon$-Lie algebras}

Consider the second version of an obscure double $\varepsilon_{\pi\mu
}^{\left(  2\right)  }/\mathbf{\epsilon}_{\mu}^{\left(  2\right)  }%
$-commutative algebra $\mathcal{A}_{\mathcal{G}\circledast}\left(  \mu\right)
$ defined in (\ref{ab1})--(\ref{em3}) and construct a corresponding analog of
the Lie algebra, following the same procedure as for associative $\varepsilon
$-commutative algebras \cite{sch79,rit/wyl,mon97}.

Take $\mathcal{A}_{\mathcal{G}\circledast}\left(  \mu\right)  $ and define a
\textit{double }$\varepsilon\epsilon$\textit{-Lie bracket} $\mathit{L}%
_{\varepsilon\epsilon}:\mathfrak{A}^{\left(  \mu\right)  }\otimes
\mathfrak{A}^{\left(  \mu\right)  }\rightarrow\mathfrak{A}^{\left(
\mu\right)  }$ by%
\begin{equation}
\mathit{L}_{\varepsilon\epsilon}\left[  a,b\right]  =a\circledast
b-\varepsilon_{\pi\mu}^{\left(  2\right)  }\left(  a^{\prime},b^{\prime
}\right)  \mathbf{\epsilon}_{\mu}^{\left(  2\right)  }\left(  a,b\right)
b\circledast a,\ \ \ \ \forall a,b\in\mathfrak{A}^{\left(  \mu\right)
},\ \ \forall a^{\prime},b^{\prime}\in G, \label{lab}%
\end{equation}
where $\varepsilon_{\pi\mu}^{\left(  2\right)  }$ and $\mathbf{\epsilon}_{\mu
}^{\left(  2\right)  }$ are given in (\ref{ab1}) and (\ref{epm}), respectively.

\begin{proposition}
The double $\varepsilon\epsilon$-Lie bracket is $\varepsilon\epsilon$-skew
commutative, i.e. it satisfies double commutativity with the commutation
factor $\left(  -\varepsilon_{\pi\mu}^{\left(  2\right)  }\mathbf{\epsilon
}_{\mu}^{\left(  2\right)  }\right)  $.
\end{proposition}

\begin{proof}
Multiply (\ref{lab}) by $\varepsilon_{\pi\mu}^{\left(  2\right)  }\left(
b^{\prime},a^{\prime}\right)  \mathbf{\epsilon}_{\mu}^{\left(  2\right)
}\left(  b,a\right)  $ and use (\ref{em1}) to obtain%
\begin{equation}
\varepsilon_{\pi\mu}^{\left(  2\right)  }\left(  b^{\prime},a^{\prime}\right)
\mathbf{\epsilon}_{\mu}^{\left(  2\right)  }\left(  b,a\right)  \mathit{L}%
_{\varepsilon\epsilon}\left[  a,b\right]  =\varepsilon_{\pi\mu}^{\left(
2\right)  }\left(  b^{\prime},a^{\prime}\right)  \mathbf{\epsilon}_{\mu
}^{\left(  2\right)  }\left(  b,a\right)  a\circledast b-b\circledast
a=-\mathit{L}_{\varepsilon\epsilon}\left[  b,a\right]  .
\end{equation}

Therefore,%
\begin{equation}
\mathit{L}_{\varepsilon\epsilon}\left[  a,b\right]  =-\varepsilon_{\pi\mu
}^{\left(  2\right)  }\left(  a^{\prime},b^{\prime}\right)  \mathbf{\epsilon
}_{\mu}^{\left(  2\right)  }\left(  a,b\right)  \mathit{L}_{\varepsilon
\epsilon}\left[  b,a\right]  , \label{leb}%
\end{equation}
which should be compared with (\ref{ab1}).
\end{proof}

\begin{proposition}
The double $\varepsilon\epsilon$-Lie bracket satisfies the membership deformed
$\varepsilon\epsilon$-Jacobi identity%
\begin{align}
&  \varepsilon_{\pi\mu}^{\left(  2\right)  }\left(  a^{\prime},b^{\prime
}\right)  \mathbf{\epsilon}_{\mu}^{\left(  2\right)  }\left(  a,b\right)
\mathit{L}_{\varepsilon\epsilon}\left[  a,\mathit{L}_{\varepsilon\epsilon
}\left[  b,c\right]  \right]  +cyclic\ permutations =0,\label{jac}\\
&  \forall a,b,c \in\mathfrak{A}^{\left(  \mu\right)  },\ \ \forall a^{\prime
},b^{\prime},c^{\prime}\in G.\nonumber
\end{align}

\end{proposition}

\begin{definition}
A \textit{double }$\varepsilon\epsilon$\textit{-Lie algebra} is an obscure
$G$-graded algebra with the double $\varepsilon\epsilon$-Lie bracket
(satisfying $\varepsilon\epsilon$-skew commutativity (\ref{leb}) and the
membership deformed Jacobi identity (\ref{jac})) as a multiplication, that
is\textit{ }$\mathcal{A}_{\mathcal{G}L}\left(  \mu\right)  =\left\langle
\mathfrak{A}^{\left(  \mu\right)  }\mid\mathit{L}_{\varepsilon\epsilon
},+\right\rangle $.
\end{definition}

\section{\label{sec-proj}\textsc{Projective representations}}

To generalize the $\varepsilon$-commutative algebras to the $n$-ary case, we
need to introduce $n$-ary projective representations and study them in brief.

\subsection{Binary projective representations}

First, briefly recall some general properties of the Schur factors and
corresponding commutation factors in the projective representation theory of
Abelian groups \cite{zmu60,zmu72} (see, also, \cite{fru55} and the unitary ray
representations \cite{bar54}). We show some known details in our notation
which can be useful in further extensions of the well-known binary
constructions to the $n$-ary case.

Let $\mathcal{H}^{\left(  2\right)  }=\left\langle H\mid\dotplus\right\rangle
\ $be a binary Abelian group and $f:\mathcal{H}^{\left(  2\right)
}\rightarrow\mathcal{E}^{\left(  2\right)  }$ , where $\mathcal{E}^{\left(
2\right)  }=\left\langle \operatorname*{End}V\mid\circ\right\rangle $, and $V$
is a vector space over a field $\Bbbk$. A map $f$ is a (binary)
\textit{projective representation} ($\sigma$-\textit{representation}
\cite{zmu60}), if $f\left(  x_{1}\right)  \circ f\left(  x_{2}\right)
=\pi_{0}^{\left(  2\right)  }\left(  x_{1},x_{2}\right)  f\left(
x_{1}\dotplus x_{2}\right)  $, $x_{1},x_{2}\in H$, and $\pi_{0}^{\left(
2\right)  }:\mathcal{H}^{\left(  2\right)  }\times\mathcal{H}^{\left(
2\right)  }\rightarrow\Bbbk^{\times}$ is a (Schur) factor, while $\left(
\circ\right)  $ is a (noncommutative binary) product in $\operatorname*{End}%
V$). The \textquotedblleft associativity relation\textquotedblright\ of
factors follows immediately from the associativity of $\left(  \circ\right)  $
such that (cf. (\ref{sa}))%
\begin{equation}
\pi_{0}^{\left(  2\right)  }\left(  x_{1},x_{2}\dotplus x_{3}\right)  \pi
_{0}^{\left(  2\right)  }\left(  x_{2},x_{3}\right)  =\pi_{0}^{\left(
2\right)  }\left(  x_{1},x_{2}\right)  \pi_{0}^{\left(  2\right)  }\left(
x_{1}\dotplus x_{2},x_{3}\right)  ,\ \ \ \forall x_{1},x_{2},x_{3}\in
H.\label{ss}%
\end{equation}

Two factor systems are equivalent $\left\{  \pi_{0}^{\left(  2\right)
}\right\}  \overset{\lambda}{\sim}\left\{  \tilde{\pi}_{0}^{\left(  2\right)
}\right\}  $ (or associated \cite{zmu72}), if there exists $\lambda
:\mathcal{H}^{\left(  2\right)  }\rightarrow\Bbbk^{\times}$ such that%
\begin{equation}
\tilde{\pi}_{0}^{\left(  2\right)  }\left(  x_{1},x_{2}\right)  =\frac
{\lambda\left(  x_{1}\right)  \lambda\left(  x_{2}\right)  }{\lambda\left(
x_{1}\dotplus x_{2}\right)  }\pi_{0}^{\left(  2\right)  }\left(  x_{1}%
,x_{2}\right)  ,\ \ \ \ \ \ \forall x_{1},x_{2}\in H,
\end{equation}
and the $\tilde{\pi}_{0}^{\left(  2\right)  }$-representation is given by
$\tilde{f}\left(  x\right)  =\lambda\left(  x\right)  f\left(  x\right)  $,
$x\in H$. The cocycle condition (\ref{ss}) means that $\pi_{0}^{\left(
2\right)  }$ belongs to the group $Z^{2}\left(  \mathcal{H}^{\left(  2\right)
},\Bbbk^{\times}\right)  $ of 2-cocycles of $\mathcal{H}^{\left(  2\right)  }$
over $\Bbbk^{\times}$, the quotient $\left\{  \pi_{0}^{\left(  2\right)
}\right\}  \diagup\overset{\lambda}{\sim}$ gives the corresponding multiplier
group, if $\Bbbk=\mathbb{C}$, and, in the general case, coincides with the
exponents of the 2-cohomology classes $H^{2}\left(  \mathcal{H}^{\left(
2\right)  },\Bbbk\right)  $ (for details, see, e.g.,
\cite{zmu60,zmu72,gou/mas/wal}).

The group $\mathcal{H}^{\left(  2\right)  }$ is Abelian, and therefore we have
(cf. (\ref{sab}))%
\begin{equation}
\pi_{0}^{\left(  2\right)  }\left(  x_{2},x_{1}\right)  f\left(  x_{1}\right)
\circ f\left(  x_{2}\right)  =\pi_{0}^{\left(  2\right)  }\left(  x_{2}%
,x_{1}\right)  \pi_{0}^{\left(  2\right)  }\left(  x_{1},x_{2}\right)
f\left(  x_{1}\dotplus x_{2}\right)  =\pi_{0}^{\left(  2\right)  }\left(
x_{1},x_{2}\right)  f\left(  x_{2}\right)  \circ f\left(  x_{1}\right)  ,
\label{sf}%
\end{equation}
which allows us to introduce a (binary) \textit{commutation factor}
$\varepsilon_{\pi_{0}}^{\left(  2\right)  }:\mathcal{H}^{\left(  2\right)
}\times\mathcal{H}^{\left(  2\right)  }\rightarrow\Bbbk^{\times}$ by (see
(\ref{ab}), (\ref{es}))%
\begin{align}
\varepsilon_{\pi_{0}}^{\left(  2\right)  }\left(  x_{1},x_{2}\right)   &
=\frac{\pi_{0}^{\left(  2\right)  }\left(  x_{1},x_{2}\right)  }{\pi
_{0}^{\left(  2\right)  }\left(  x_{2},x_{1}\right)  },\label{es1}\\
f\left(  x_{1}\right)  \circ f\left(  x_{2}\right)   &  =\varepsilon_{\pi_{0}%
}^{\left(  2\right)  }\left(  x_{1},x_{2}\right)  f\left(  x_{2}\right)  \circ
f\left(  x_{1}\right)  ,\ \forall x_{1},x_{2}\in H, \label{ff}%
\end{align}
with the obvious \textquotedblleft normalization\textquotedblright\ (cf.
(\ref{e1}))%
\begin{equation}
\varepsilon_{\pi_{0}}^{\left(  2\right)  }\left(  x_{1},x_{2}\right)
\varepsilon_{\pi_{0}}^{\left(  2\right)  }\left(  x_{2},x_{1}\right)
=1,\ \ \ \varepsilon_{\pi_{0}}^{\left(  2\right)  }\left(  x,x\right)
=1,\ \forall x_{1},x_{2},x\in H, \label{ee1}%
\end{equation}
which will be important in the $n$-ary case below.

The multiplication of the commutation factors follows from different
permutations of the three terms
\begin{align}
&  f\left(  y\right)  \circ f\left(  x_{1}\right)  \circ f\left(
x_{2}\right)  =\varepsilon_{\pi_{0}}^{\left(  2\right)  }\left(
y,x_{1}\right)  f\left(  x_{1}\right)  \circ f\left(  y\right)  \circ f\left(
x_{2}\right) \nonumber\\
&  =\varepsilon_{\pi_{0}}^{\left(  2\right)  }\left(  y,x_{1}\right)
\varepsilon_{\pi_{0}}^{\left(  2\right)  }\left(  y,x_{2}\right)  f\left(
x_{1}\right)  \circ f\left(  x_{2}\right)  \circ f\left(  y\right) \nonumber\\
&  =\pi_{0}^{\left(  2\right)  }\left(  x_{1},x_{2}\right)  f\left(  y\right)
\circ f\left(  x_{1}\dotplus x_{2}\right)  =\varepsilon_{\pi_{0}}^{\left(
2\right)  }\left(  y,x_{1}\dotplus x_{2}\right)  \left(  \pi_{0}^{\left(
2\right)  }\left(  x_{1},x_{2}\right)  f\left(  x_{1}\dotplus x_{2}\right)
\right)  \circ f\left(  y\right) \nonumber\\
&  =\varepsilon_{\pi_{0}}^{\left(  2\right)  }\left(  y,x_{1}\dotplus
x_{2}\right)  f\left(  x_{1}\right)  \circ f\left(  x_{2}\right)  \circ
f\left(  y\right)  ,\ \ \ \forall x_{1},x_{2},y\in H. \label{fy}%
\end{align}

Thus, it follows that the commutation factor multiplication is (and similarly
for the second place, cf. (\ref{e2})-(\ref{e3}))%
\begin{equation}
\varepsilon_{\pi_{0}}^{\left(  2\right)  }\left(  y,x_{1}\dotplus
x_{2}\right)  =\varepsilon_{\pi_{0}}^{\left(  2\right)  }\left(
y,x_{1}\right)  \varepsilon_{\pi_{0}}^{\left(  2\right)  }\left(
y,x_{2}\right)  ,\ \ \ \forall x_{1},x_{2},y\in H, \label{eee}%
\end{equation}
which means that $\varepsilon_{\pi_{0}}^{\left(  2\right)  }$ is a (binary)
bicharacter on $\mathcal{H}^{\left(  2\right)  }$, because for $\chi
_{y}^{\left(  2\right)  }\left(  x\right)  \equiv\varepsilon_{\pi_{0}%
}^{\left(  2\right)  }\left(  y,x\right)  $ we have%
\begin{equation}
\chi_{y}^{\left(  2\right)  }\left(  x_{1}\right)  \chi_{y}^{\left(  2\right)
}\left(  x_{2}\right)  =\chi_{y}^{\left(  2\right)  }\left(  x_{1}\dotplus
x_{2}\right)  ,\ \ \forall x_{1},x_{2}\in H. \label{xx}%
\end{equation}

Denoting the group of bicharacters $\chi_{y}^{\left(  2\right)  }$ on
$\mathcal{H}^{\left(  2\right)  }$ with the multiplication (\ref{xx}) by
$\mathcal{B}^{\left(  2\right)  }(\mathcal{H}^{\left(  2\right)  },\Bbbk)$, we
observe that the mapping $\pi_{0}^{\left(  2\right)  }\rightarrow
\varepsilon_{\pi_{0}}^{\left(  2\right)  }$ is a homomorphism of $Z^{2}\left(
\mathcal{H}^{\left(  2\right)  },\Bbbk^{\times}\right)  $ to $\mathcal{B}%
^{\left(  2\right)  }(\mathcal{H}^{\left(  2\right)  },\Bbbk)$ with the kernel
a subgroup of the 2-coboundaries of $\mathcal{H}^{\left(  2\right)  }$ over
$\Bbbk^{\times}$ (for more details, see \cite{zmu72}).

\subsection{$n$-ary projective representations}

Here we consider some features of $n$-ary projective representations and
corresponding particular generalizations of binary $\varepsilon$-commutativity.

Let $\mathcal{H}^{\left(  n\right)  }=\left\langle H\mid\left[  \ \ \right]
_{\dotplus}^{\left(  n\right)  }\right\rangle \ $ be an $n$-ary Abelian group
with the totally commutative multiplication $\left[  \ \ \right]  _{\dotplus
}^{\left(  n\right)  }$, and the mapping $f:\mathcal{H}^{\left(  n\right)
}\rightarrow\mathcal{E}^{\left(  n\right)  }$ , where $\mathcal{E}^{\left(
n\right)  }=\left\langle \operatorname*{End}V\mid\left[  \ \ \right]  _{\circ
}^{\left(  n\right)  }\right\rangle $ (for general polyadic representations,
see \cite{dup2018a} and refs therein). Here we suppose that $V$ is a vector
space over a field $\Bbbk$, and $\left[  \ \ \right]  _{\circ}^{\left(
n\right)  }$ is an $n$-ary associative product in $\operatorname*{End}V$,
which means that%
\begin{align}
&  f\left[  \left[  f\left(  x_{1}\right)  ,\ldots,f\left(  x_{n}\right)
\right]  _{\circ}^{\left(  n\right)  },f\left(  x_{n+1}\right)  ,\ldots
,f\left(  x_{2n-1}\right)  \right]  _{\circ}^{\left(  n\right)  }\nonumber\\
&  =f\left[  f\left(  x_{1}\right)  ,\left[  f\left(  x_{2}\right)
,\ldots,f\left(  x_{n+1}\right)  \right]  _{\circ}^{\left(  n\right)
},f\left(  x_{n+2}\right)  ,\ldots,f\left(  x_{2n-1}\right)  \right]  _{\circ
}^{\left(  n\right)  }\nonumber\\
&  \vdots\nonumber\\
&  =f\left[  f\left(  x_{1}\right)  ,f\left(  x_{2}\right)  ,\ldots,f\left(
x_{n-1}\right)  ,\left[  f\left(  x_{n}\right)  ,\ldots,f\left(
x_{2n-1}\right)  \right]  _{\circ}^{\left(  n\right)  }\right]  _{\circ
}^{\left(  n\right)  }\label{f}%
\end{align}

\begin{definition}
A map $f$ is an $n$-\textit{ary} \textit{projective representation}, if%
\begin{equation}
\left[  f\left(  x_{1}\right)  ,\ldots,f\left(  x_{n}\right)  \right]
_{\circ}^{\left(  n\right)  }=\pi_{0}^{\left(  n\right)  }\left(  x_{1}%
,\ldots,x_{n}\right)  f\left(  \left[  x_{1},\ldots,x_{n}\right]  _{\dotplus
}^{\left(  n\right)  }\right)  ,\ \ \ x_{1},\ldots,x_{n}\in H, \label{fn}%
\end{equation}
and $\pi_{0}^{\left(  n\right)  }:\overset{n}{\overbrace{\mathcal{H}^{\left(
n\right)  }\times\ldots\times\mathcal{H}^{\left(  n\right)  }}}\rightarrow
\Bbbk^{\times}$ is an $n$-\textit{ary (Schur-like) factor}.
\end{definition}

\begin{proposition}
The factors $\pi_{0}^{\left(  n\right)  }$ satisfy the $n$-ary $2$-cocycle
conditions (cf. (\ref{ss}))%
\begin{align}
&  \pi_{0}^{\left(  n\right)  }\left(  x_{1},\ldots,x_{n}\right)  \pi
_{0}^{\left(  n\right)  }\left(  \left[  x_{1},\ldots,x_{n}\right]
_{\dotplus}^{\left(  n\right)  },x_{n+1},\ldots,x_{2n-1}\right) \nonumber\\
&  =\pi_{0}^{\left(  n\right)  }\left(  x_{2},\ldots,x_{n+1}\right)  \pi
_{0}^{\left(  n\right)  }\left(  x_{1},\left[  x_{2},\ldots,x_{n+1}\right]
_{\dotplus}^{\left(  n\right)  },x_{n+2},\ldots,x_{2n-1}\right) \nonumber\\
&  \vdots\nonumber\\
&  =\pi_{0}^{\left(  n\right)  }\left(  x_{n},\ldots,x_{n+1}\right)  \pi
_{0}^{\left(  n\right)  }\left(  x_{1},\ldots,x_{n-1},\left[  x_{n}%
,\ldots,x_{2n-1}\right]  _{\dotplus}^{\left(  n\right)  }\right)
,\ \ x_{1},\ldots,x_{2n-1}\in H. \label{pp}%
\end{align}

\end{proposition}

\begin{proof}
These immediately follow from the $n$-ary associativity in
$\operatorname*{End}V$ (\ref{f}) and (\ref{fn}).
\end{proof}

Two $n$-ary factor systems are equivalent $\left\{  \pi_{0}^{\left(  n\right)
}\right\}  \overset{\lambda}{\sim}\left\{  \tilde{\pi}_{0}^{\left(  n\right)
}\right\}  $, if there exists $\lambda:\mathcal{H}^{\left(  n\right)
}\rightarrow\Bbbk^{\times}$ such that%
\begin{equation}
\tilde{\pi}_{0}^{\left(  n\right)  }\left(  x_{1},\ldots,x_{n}\right)
=\frac{\lambda\left(  x_{1}\right)  \lambda\left(  x_{2}\right)  \ldots
\lambda\left(  x_{n}\right)  }{\lambda\left(  \left[  x_{1},\ldots
,x_{n}\right]  _{\dotplus}^{\left(  n\right)  }\right)  }\pi_{0}^{\left(
n\right)  }\left(  x_{1},\ldots,x_{n}\right)  ,\ \ \ \ \ \ \forall x_{1}%
,x_{2}\in H, \label{pn}%
\end{equation}
and the $\tilde{\pi}_{0}^{\left(  n\right)  }$-representation is given by
$\tilde{f}\left(  x\right)  =\lambda\left(  x\right)  f\left(  x\right)  $,
$x\in H$.

To understand how properly and uniquely to introduce the commutation factors
for $n$-ary projective representations, we need to consider an $n$-ary analog
of (\ref{sf}) (see, also (\ref{sab})).

\begin{proposition}
The commutativity of a $\pi_{0}^{\left(  n\right)  }$-representation is given
by $\left(  n!-1\right)  $ relations of the form%
\begin{equation}
\pi_{0}^{\left(  n\right)  }\left(  x_{\sigma\left(  1\right)  }%
,\ldots,x_{\sigma\left(  n\right)  }\right)  \left[  f\left(  x_{1}\right)
,\ldots,f\left(  x_{n}\right)  \right]  _{\circ}^{\left(  n\right)  }=\pi
_{0}^{\left(  n\right)  }\left(  x_{1},\ldots,x_{n}\right)  \left[  f\left(
x_{\sigma\left(  1\right)  }\right)  ,\ldots,f\left(  x_{\sigma\left(
n\right)  }\right)  \right]  _{\circ}^{\left(  n\right)  }, \label{pf}%
\end{equation}
where $\sigma\in S_{n}$, $\sigma\neq I$, $S_{n}$ is the symmetry permutation
group on $n$ elements.
\end{proposition}

\begin{proof}
Using the definition of the $n$-ary projective representation (\ref{fn}), we
obtain for the l.h.s. and r.h.s. of (\ref{pf})%
\begin{equation}
\pi_{0}^{\left(  n\right)  }\left(  x_{\sigma\left(  1\right)  }%
,\ldots,x_{\sigma\left(  n\right)  }\right)  \pi_{0}^{\left(  n\right)
}\left(  x_{1},\ldots,x_{n}\right)  f\left(  \left[  x_{1},\ldots
,x_{n}\right]  _{\dotplus}^{\left(  n\right)  }\right)
\end{equation}
and%
\begin{equation}
\pi_{0}^{\left(  n\right)  }\left(  x_{1},\ldots,x_{n}\right)  \pi
_{0}^{\left(  n\right)  }\left(  x_{\sigma\left(  1\right)  },\ldots
,x_{\sigma\left(  n\right)  }\right)  f\left(  \left[  x_{\sigma\left(
1\right)  },\ldots,x_{\sigma\left(  n\right)  }\right]  _{\dotplus}^{\left(
n\right)  }\right)  ,
\end{equation}
respectively.

Because $\mathcal{H}^{\left(  n\right)  }$ is totally commutative
\begin{equation}
\left[  x_{1},\ldots,x_{n}\right]  _{\dotplus}^{\left(  n\right)  }=\left[
x_{\sigma\left(  1\right)  },\ldots,x_{\sigma\left(  n\right)  }\right]
_{\dotplus}^{\left(  n\right)  },\ \ \forall x_{1},\ldots,x_{n}\in
H,\forall\sigma\in S_{n}.\label{fcom}%
\end{equation}
then $f\left(  \left[  x_{1},\ldots,x_{n}\right]  _{\dotplus}^{\left(
n\right)  }\right)  =f\left(  \left[  x_{\sigma\left(  1\right)  }%
,\ldots,x_{\sigma\left(  n\right)  }\right]  _{\dotplus}^{\left(  n\right)
}\right)  $. Taking into account all nonidentical permutations, we get
$\left(  n!-1\right)  $ relations in (\ref{pf}).
\end{proof}

\begin{corollary}
To describe the noncommutativity of an $n$-ary projective representation, we
need to have \textsf{not one} relation between Schur factors (as in the binary
case (\ref{sf}) and (\ref{sab})), but $\left(  n!-1\right)  $
\textsf{relations} (\ref{pf}). This leads to the concept of the set of
$\left(  n!-1\right)  $ commutation factors.
\end{corollary}

\begin{definition}
The commutativity of the $n$-ary projective representation with the Schur-like
factor $\pi_{0}^{\left(  n\right)  }\left(  x_{1},\ldots,x_{n}\right)  $ is
governed by the \textit{set of} $\left(  n!-1\right)  $ \textit{commutation
factors}%
\begin{align}
\varepsilon_{\sigma\left(  1\right)  ,\ldots,\sigma\left(  n\right)
}^{\left(  n\right)  }  &  \equiv\varepsilon_{\sigma\left(  1\right)
,\ldots,\sigma\left(  n\right)  }^{\left(  n\right)  }\left(  x_{1}%
,\ldots,x_{n}\right)  =\frac{\pi_{0}^{\left(  n\right)  }\left(  x_{1}%
,\ldots,x_{n}\right)  }{\pi_{0}^{\left(  n\right)  }\left(  x_{\sigma\left(
1\right)  },\ldots,x_{\sigma\left(  n\right)  }\right)  },\label{ens}\\
\left[  f\left(  x_{1}\right)  ,\ldots,f\left(  x_{n}\right)  \right]
_{\circ}^{\left(  n\right)  }  &  =\varepsilon_{\sigma\left(  1\right)
,\ldots,\sigma\left(  n\right)  }^{\left(  n\right)  }\left[  f\left(
x_{\sigma\left(  1\right)  }\right)  ,\ldots,f\left(  x_{\sigma\left(
n\right)  }\right)  \right]  _{\circ}^{\left(  n\right)  }, \label{fef}%
\end{align}
where $\sigma\in S_{n}$. We call all $\varepsilon_{\sigma\left(  1\right)
,\ldots,\sigma\left(  n\right)  }^{\left(  n\right)  }$ as a
\textit{commutation factor }\textquotedblleft\textit{vector}\textquotedblright%
\ and denote it by $\vec{\varepsilon}$, where its components will be
enumerated by lexicographic order.
\end{definition}

Thus, each component of $\vec{\varepsilon}$ is responsible for the commutation
of any two $n$-ary monomials with different permutations $\sigma
,\sigma^{\prime}\in S_{n}$, since from (\ref{fef}) it follows that%
\begin{equation}
\left[  f\left(  x_{\sigma^{\prime}\left(  1\right)  }\right)  ,\ldots
,f\left(  x_{\sigma^{\prime}\left(  n\right)  }\right)  \right]  _{\circ
}^{\left(  n\right)  }=\varepsilon_{\sigma\left(  1\right)  ,\ldots
,\sigma\left(  n\right)  }^{\left(  n\right)  }\left(  x_{\sigma^{\prime
}\left(  1\right)  },\ldots,x_{\sigma^{\prime}\left(  n\right)  }\right)
\left[  f\left(  x_{\sigma\left(  1\right)  }\right)  ,\ldots,f\left(
x_{\sigma\left(  n\right)  }\right)  \right]  _{\circ}^{\left(  n\right)  },
\end{equation}

It follows from (\ref{ens}) that an $n$-ary analog of the normalization
property (\ref{ee1}) is%
\begin{align}
\varepsilon_{\sigma\left(  1\right)  ,\ldots,\sigma\left(  n\right)
}^{\left(  n\right)  }\left(  x_{\sigma^{\prime}\left(  1\right)  }%
,\ldots,x_{\sigma^{\prime}\left(  n\right)  }\right)  \varepsilon
_{\sigma^{\prime}\left(  1\right)  ,\ldots,\sigma^{\prime}\left(  n\right)
}^{\left(  n\right)  }\left(  x_{\sigma\left(  1\right)  },\ldots
,x_{\sigma\left(  n\right)  }\right)   &  =1,\\
\varepsilon_{1,\ldots,n}^{\left(  n\right)  }\left(  x,\ldots,x\right)   &
=1,
\end{align}
where $\sigma$,$\sigma^{\prime}\in S_{n}$, $x_{1},\ldots,x_{n},x\in H$.

In this notation the binary commutation factor (\ref{es1}) is $\varepsilon
_{\pi_{0}}^{\left(  2\right)  }\left(  x_{1},x_{2}\right)  =\varepsilon
_{21}^{\left(  2\right)  }\left(  x_{1},x_{2}\right)  $.

\begin{example}
[\textsf{Ternary projective representation}]Consider the minimal nonbinary
case $n=3$. The ternary projective representation of the Abelian ternary group
$\mathcal{H}^{\left(  3\right)  }$ is given by%
\begin{equation}
\left[  f\left(  x_{1}\right)  ,f\left(  x_{2}\right)  ,f\left(  x_{3}\right)
\right]  =\pi_{0}\left(  x_{1},x_{2},x_{3}\right)  f\left(  x_{1}\dotplus
x_{2}\dotplus x_{3}\right)  ,\ \ \ x_{1},x_{2},x_{3}\in H,
\end{equation}
where we denote $\pi_{0}^{\left(  3\right)  }\equiv\pi_{0}$, $\left[
\ \ \right]  _{\circ}^{\left(  3\right)  }\equiv\left[  \ \ \right]  $ and
$\left[  x_{1},x_{2},x_{3}\right]  _{\dotplus}^{\left(  3\right)  }\equiv
x_{1}\dotplus x_{2}\dotplus x_{3}$.

The ternary 2-cocycle conditions for the ternary Schur-like factor $\pi_{0}$
now become%
\begin{align}
&  \pi_{0}\left(  x_{1},x_{2},x_{3}\right)  \pi_{0}\left(  x_{1}\dotplus
x_{2}\dotplus x_{3},x_{4},x_{5}\right) \nonumber\\
&  =\pi_{0}\left(  x_{2},x_{3},x_{4}\right)  \pi_{0}\left(  x_{1}%
,x_{2}\dotplus x_{3}\dotplus x_{4},x_{5}\right) \nonumber\\
&  =\pi_{0}\left(  x_{3},x_{4},x_{5}\right)  \pi_{0}\left(  x_{1},x_{2}%
,x_{3}\dotplus x_{4}\dotplus x_{5}\right)  . \label{px0}%
\end{align}

Thus, we obtain $\left(  3!-1\right)  =5$ different ternary commutation
relations and the corresponding $5$-dimensional \textquotedblleft
vector\textquotedblright\ $\vec{\varepsilon}$%
\begin{align}
&  \left[  f\left(  x_{1}\right)  ,f\left(  x_{2}\right)  ,f\left(
x_{3}\right)  \right]  =\varepsilon_{\sigma\left(  1\right)  \sigma\left(
2\right)  \sigma\left(  3\right)  }\left[  f\left(  x_{\sigma\left(  1\right)
}\right)  ,f\left(  x_{\sigma\left(  2\right)  }\right)  ,f\left(
x_{\sigma\left(  3\right)  }\right)  \right]  ,\\
\varepsilon_{132}  &  =\frac{\pi_{0}\left(  x_{1},x_{2},x_{3}\right)  }%
{\pi_{0}\left(  x_{1},x_{3},x_{2}\right)  },\ \ \varepsilon_{231}=\frac
{\pi_{0}\left(  x_{1},x_{2},x_{3}\right)  }{\pi_{0}\left(  x_{2},x_{3}%
,x_{1}\right)  },\ \ \varepsilon_{213}=\frac{\pi_{0}\left(  x_{1},x_{2}%
,x_{3}\right)  }{\pi_{0}\left(  x_{2},x_{1},x_{3}\right)  },\nonumber\\
\varepsilon_{312}  &  =\frac{\pi_{0}\left(  x_{1},x_{2},x_{3}\right)  }%
{\pi_{0}\left(  x_{3},x_{1},x_{2}\right)  },\ \ \varepsilon_{321}=\frac
{\pi_{0}\left(  x_{1},x_{2},x_{3}\right)  }{\pi_{0}\left(  x_{3},x_{2}%
,x_{1}\right)  },\ \ x_{1},x_{2},x_{3}\in H,\ \sigma\in S_{3}. \label{ep0}%
\end{align}

\end{example}

\section{$n$\textsc{-ary double commutative algebras}}

\subsection{$n$-ary $\varepsilon$-commutative algebras}

Here we introduce grading noncommutativity for $n$-ary algebras
(\textquotedblleft$n$-ary coloring\textquotedblright), which is closest to the
binary (\textquotedblleft coloring\textquotedblright) case (\ref{ab}). We are
exploiting an $n$-ary analog of the (Schur) factor (\ref{es}) and its relation
(\ref{sab}) by means of the $n$-ary projective representation theory from
\textsc{Section} \ref{sec-proj}.

Let $\mathcal{A}^{\left(  n\right)  }=\left\langle A\mid\left[  \ \right]
^{\left(  n\right)  },+\right\rangle $ be an associative $n$-ary algebra
\cite{car4,mic/vin,oli60} over a binary field $\Bbbk$ (with the binary
addition) having zero $z\in A$ and unit $e\in A$, if $\mathcal{A}^{\left(
n\right)  }$ is unital (for polyadic algebras with all nonbinary operations,
see \cite{dup2019}). A $n$-ary \textit{graded algebra} $\mathcal{A}%
_{\mathcal{G}}^{\left(  n\right)  }$ (an $n$-ary $G$-graded $\Bbbk$-algebra)
is a direct sum of subalgebras $\mathcal{A}_{\mathcal{G}}^{\left(  n\right)
}=\bigoplus_{g\in G}\mathcal{A}_{g}$, where $\mathcal{G}=\left\langle
G\mid+^{\prime}\right\rangle $ is a (binary Abelian) \textit{grading group}
and the set $n$-ary multiplication \textquotedblleft respects the
gradation\textquotedblright%
\begin{equation}
\left[  A_{g_{1}},\ldots,A_{g_{n}}\right]  ^{\left(  n\right)  }\subseteq
A_{g_{1}+^{\prime}\ldots+^{\prime}g_{n}},\ \ \ \ g_{1},\ldots,g_{n}\in G.
\end{equation}

As in the binary case (\ref{i}) the elements from $A_{g}\subset A$ are
homogeneous of degree $a^{\prime}=g\in G$.

It is natural to start our $n$-ary consideration from the Schur-like factors
(\ref{pp}) which generalize (\ref{sa}) and (\ref{ss}).

\begin{definition}
In an $n$-ary graded algebra $\mathcal{A}_{\mathcal{G}}^{\left(  n\right)  }$
the $n$\textit{-ary Schur-like factor }is an $n$-place function on gradings
$\pi^{\left(  n\right)  }:\overset{n}{\overbrace{G\times\ldots\times G}%
}\rightarrow A$ satisfying the $n$-ary cocycle condition (cf. (\ref{sa}) and
(\ref{pp}))%
\begin{align}
&  \pi^{\left(  n\right)  }\left(  a_{1}^{\prime},\ldots,a_{n}^{\prime
}\right)  \pi^{\left(  n\right)  }\left(  \left[  a_{1}^{\prime},\ldots
,a_{n}^{\prime}\right]  ^{\left(  n\right)  },a_{n+1}^{\prime},\ldots
,a_{2n-1}^{\prime}\right) \nonumber\\
&  =\pi^{\left(  n\right)  }\left(  a_{2}^{\prime},\ldots,a_{n+1}^{\prime
}\right)  \pi^{\left(  n\right)  }\left(  a_{1}^{\prime},\left[  a_{2}%
^{\prime},\ldots,a_{n+1}^{\prime}\right]  ^{\left(  n\right)  },a_{n+2}%
^{\prime},\ldots,a_{2n-1}^{\prime}\right) \nonumber\\
&  \vdots\nonumber\\
&  =\pi^{\left(  n\right)  }\left(  a_{n}^{\prime},\ldots,a_{n+1}^{\prime
}\right)  \pi^{\left(  n\right)  }\left(  a_{1}^{\prime},\ldots,a_{n-1}%
^{\prime},\left[  a_{n}^{\prime},\ldots,a_{2n-1}^{\prime}\right]  ^{\left(
n\right)  }\right)  ,\ \ a_{1}^{\prime},\ldots,a_{2n-1}^{\prime}\in G.
\label{pa}%
\end{align}

\end{definition}

There are many possible ways to introduce noncommutativity for $n$-ary
algebras \cite{azc/izq,mic/vin}. We propose a \textquotedblleft projective
version\textquotedblright\ of noncommutativity in $\mathcal{A}_{\mathcal{G}%
}^{\left(  n\right)  }$ which naturally follows from the $n$-ary projective
representations (\ref{ens}) and can be formulated in terms of the Schur-like
factors as in (\ref{sab}) and (\ref{pf}).

\begin{definition}
An $n$-ary graded algebra $\mathcal{A}_{\mathcal{G}}^{\left(  n\right)  }$ is
called $\pi$\textit{-commutative}, if the $\left(  n!-1\right)  $ relations
(cf. (\ref{pf}))%
\begin{equation}
\pi^{\left(  n\right)  }\left(  a_{\sigma\left(  1\right)  }^{\prime}%
,\ldots,a_{\sigma\left(  n\right)  }^{\prime}\right)  \left[  a_{1}%
,\ldots,a_{n}\right]  ^{\left(  n\right)  }=\pi^{\left(  n\right)  }\left(
a_{1}^{\prime},\ldots,a_{n}^{\prime}\right)  \left[  a_{\sigma\left(
1\right)  },\ldots,a_{\sigma\left(  n\right)  }\right]  ^{\left(  n\right)  },
\label{paa}%
\end{equation}
hold for all $a_{1},\ldots,a_{n}\in A$, and $\sigma\in S_{n}$, $\sigma\neq I$,
where $\pi^{\left(  n\right)  }$ are $n$-ary Schur-like factors satisfying
(\ref{pa}).
\end{definition}

Two $n$-ary Schur-like factor systems are equivalent $\left\{  \pi^{\left(
n\right)  }\right\}  \overset{\tilde{\lambda}}{\sim}\left\{  \tilde{\pi
}^{\left(  n\right)  }\right\}  $, if there exists $\tilde{\lambda
}:\mathcal{A}_{\mathcal{G}}^{\left(  n\right)  }\rightarrow\Bbbk^{\times}$
such that (cf. (\ref{pn}))%
\begin{equation}
\tilde{\pi}^{\left(  n\right)  }\left(  a_{1}^{\prime},\ldots,a_{n}^{\prime
}\right)  =\frac{\tilde{\lambda}\left(  a_{1}^{\prime}\right)  \ldots
\tilde{\lambda}\left(  a_{n}^{\prime}\right)  }{\tilde{\lambda}\left(  \left[
a_{1}^{\prime},\ldots,a_{n}^{\prime}\right]  ^{\left(  n\right)  }\right)
}\pi^{\left(  n\right)  }\left(  a_{1}^{\prime},\ldots,a_{n}^{\prime}\right)
,\ \ \ \ \ \ \forall a_{1}^{\prime},\ldots,a_{2}^{\prime}\in G.
\end{equation}

The quotient by this equivalence relation $\left\{  \pi^{\left(  n\right)
}\right\}  \diagup\overset{\tilde{\lambda}}{\sim}$ is the corresponding
multiplier group, as in the binary case \cite{sch79}.

Let $\varphi\in\operatorname*{Aut}G$ and $\left\{  \pi^{\left(  n\right)
}\right\}  $ be a factor system, then its pullback%
\begin{equation}
\pi_{\ast}^{\left(  n\right)  }\left(  a_{1}^{\prime},\ldots,a_{n}^{\prime
}\right)  =\pi^{\left(  n\right)  }\left(  \varphi\left(  a_{1}^{\prime
}\right)  ,\ldots,\varphi\left(  a_{n}^{\prime}\right)  \right)  \label{ps}%
\end{equation}
is also a factor system $\left\{  \pi_{\ast}^{\left(  n\right)  }\right\}  $.
In the $n$-ary case the \textquotedblleft homotopic\textquotedblright\ analog
of (\ref{ps}) is possible%
\begin{equation}
\pi_{\ast\ast}^{\left(  n\right)  }\left(  a_{1}^{\prime},\ldots,a_{n}%
^{\prime}\right)  =\pi^{\left(  n\right)  }\left(  \varphi_{1}\left(
a_{1}^{\prime}\right)  ,\ldots,\varphi_{n}\left(  a_{n}^{\prime}\right)
\right)  , \label{pss}%
\end{equation}
where $\varphi_{1},\ldots,\varphi_{n}\in\operatorname*{Aut}G$, such that
$\left\{  \pi_{\ast\ast}^{\left(  n\right)  }\right\}  $ is also a factor system.

Comparing (\ref{paa}) with binary $\pi$-commutativity (\ref{sab}) we observe
that the most general description of $n$-ary graded algebras $\mathcal{A}%
_{\mathcal{G}}^{\left(  n\right)  }$ can be achieved by using at least
$\left(  n!-1\right)  $ commutation factors.

\begin{definition}
An $n$-ary graded algebra $\mathcal{A}_{\mathcal{G}}^{\left(  n\right)  }$ is
$\varepsilon^{\left(  n\right)  }$\textit{-commutative} if there are $\left(
n!-1\right)  $ commutation factors%
\begin{align}
\varepsilon_{\sigma\left(  1\right)  ,\ldots,\sigma\left(  n\right)
}^{\left(  n\right)  }  &  \equiv\varepsilon_{\sigma\left(  1\right)
,\ldots,\sigma\left(  n\right)  }^{\left(  n\right)  }\left(  a_{1}^{\prime
},\ldots,a_{n}^{\prime}\right)  =\frac{\pi^{\left(  n\right)  }\left(
a_{1}^{\prime},\ldots,a_{n}^{\prime}\right)  }{\pi^{\left(  n\right)  }\left(
a_{\sigma\left(  1\right)  }^{\prime},\ldots,a_{\sigma\left(  n\right)
}^{\prime}\right)  },\label{en}\\
\left[  a_{1},\ldots,a_{n}\right]  ^{\left(  n\right)  }  &  =\varepsilon
_{\sigma\left(  1\right)  ,\ldots,\sigma\left(  n\right)  }^{\left(  n\right)
}\left[  a_{\sigma\left(  1\right)  },\ldots,a_{\sigma\left(  n\right)
}\right]  ^{\left(  n\right)  }, \label{an}%
\end{align}
where $\sigma\in S_{n}$, $\sigma\neq I$, and the factors $\pi^{\left(
n\right)  }$ satisfy (\ref{pa}). The set of $\varepsilon_{\sigma\left(
1\right)  ,\ldots,\sigma\left(  n\right)  }^{\left(  n\right)  }$ is a
\textit{commutation factor }\textquotedblleft\textit{vector}\textquotedblright%
\ $\overrightarrow{\varepsilon}$ of the algebra $\mathcal{A}_{\mathcal{G}%
}^{\left(  n\right)  }$ having $\left(  n!-1\right)  $ components.
\end{definition}

Each component of the \textquotedblleft vector\textquotedblright%
\ $\overrightarrow{\varepsilon}$ is responsible for the commutation of two
$n$-ary monomials in $\mathcal{A}_{\mathcal{G}}^{\left(  n\right)  }$ such
that from (\ref{an}) we have%
\begin{equation}
\left[  a_{\sigma^{\prime}\left(  1\right)  },\ldots,a_{\sigma^{\prime}\left(
n\right)  }\right]  ^{\left(  n\right)  }=\varepsilon_{\sigma\left(  1\right)
,\ldots,\sigma\left(  n\right)  }^{\left(  n\right)  }\left(  a_{\sigma
^{\prime}\left(  1\right)  }^{\prime},\ldots,a_{\sigma^{\prime}\left(
n\right)  }^{\prime}\right)  \left[  a_{\sigma\left(  1\right)  }%
,\ldots,a_{\sigma\left(  n\right)  }\right]  ^{\left(  n\right)  },
\end{equation}
with permutations $\sigma,\sigma^{\prime}\in S_{n}$.

It follows from (\ref{ens}) that an $n$-ary analog of the normalization
property (\ref{ee1}) is%
\begin{align}
\varepsilon_{\sigma\left(  1\right)  ,\ldots,\sigma\left(  n\right)
}^{\left(  n\right)  }\left(  a_{\sigma^{\prime}\left(  1\right)  }^{\prime
},\ldots,a_{\sigma^{\prime}\left(  n\right)  }^{\prime}\right)  \varepsilon
_{\sigma^{\prime}\left(  1\right)  ,\ldots,\sigma^{\prime}\left(  n\right)
}^{\left(  n\right)  }\left(  a_{\sigma\left(  1\right)  }^{\prime}%
,\ldots,a_{\sigma\left(  n\right)  }^{\prime}\right)   &  =1,\\
\varepsilon_{1,\ldots,n}^{\left(  n\right)  }\left(  a^{\prime},\ldots
,a^{\prime}\right)   &  =1,
\end{align}
where $\sigma$,$\sigma^{\prime}\in S_{n}$, $a_{1}^{\prime},\ldots
,a_{n}^{\prime},a^{\prime}\in G$.

If $\varphi\in\operatorname*{Aut}G$ and $\left\{  \pi_{\ast}^{\left(
n\right)  }\right\}  $ is a pullback of $\left\{  \pi^{\left(  n\right)
}\right\}  $ (\ref{ps}), then the corresponding%
\begin{equation}
\varepsilon_{\ast\sigma\left(  1\right)  ,\ldots,\sigma\left(  n\right)
}^{\left(  n\right)  }\left(  a_{1}^{\prime},\ldots,a_{n}^{\prime}\right)
=\frac{\pi_{\ast}^{\left(  n\right)  }\left(  a_{1}^{\prime},\ldots
,a_{n}^{\prime}\right)  }{\pi_{\ast}^{\left(  n\right)  }\left(
a_{\sigma\left(  1\right)  }^{\prime},\ldots,a_{\sigma\left(  n\right)
}^{\prime}\right)  }%
\end{equation}
are commutation factors as well (if $\Bbbk$ is algebraically closed, by
analogy with \cite{sch79}). The same is true for their \textquotedblleft
homotopic\textquotedblright\ analog $\left\{  \pi_{\ast\ast}^{\left(
n\right)  }\right\}  $ (\ref{pss}).

\begin{definition}
A factor set $\left\{  \pi^{\left(  n\right)  }\right\}  $ is called
\textit{totally symmetric}, if $\overrightarrow{\varepsilon}$ is the unit
commutation factor \textquotedblleft vector\textquotedblright, such that all
of its $\left(  n!-1\right)  $ components are identities (in $\Bbbk$)%
\begin{equation}
\varepsilon_{\sigma\left(  1\right)  ,\ldots,\sigma\left(  n\right)
}^{\left(  n\right)  }\left(  a_{1}^{\prime},\ldots,a_{n}^{\prime}\right)
=1,\ \ \sigma\in S_{n},\sigma\neq I,a_{1}^{\prime},\ldots,a_{n}^{\prime}\in G.
\label{tsym}%
\end{equation}

\end{definition}

Suppose we have two factor sets $\left\{  \pi_{1}^{\left(  n\right)
}\right\}  $ and $\left\{  \pi_{2}^{\left(  n\right)  }\right\}  $ which
correspond to the same commutation factor $\varepsilon^{\left(  n\right)  }$
in $\mathcal{A}_{\mathcal{G}}^{\left(  n\right)  }$ such that%
\begin{equation}
\varepsilon_{\sigma\left(  1\right)  ,\ldots,\sigma\left(  n\right)
}^{\left(  n\right)  }\left(  a_{1}^{\prime},\ldots,a_{n}^{\prime}\right)
=\frac{\pi_{1}^{\left(  n\right)  }\left(  a_{1}^{\prime},\ldots,a_{n}%
^{\prime}\right)  }{\pi_{1}^{\left(  n\right)  }\left(  a_{\sigma\left(
1\right)  }^{\prime},\ldots,a_{\sigma\left(  n\right)  }^{\prime}\right)
}=\frac{\pi_{2}^{\left(  n\right)  }\left(  a_{1}^{\prime},\ldots
,a_{n}^{\prime}\right)  }{\pi_{2}^{\left(  n\right)  }\left(  a_{\sigma\left(
1\right)  }^{\prime},\ldots,a_{\sigma\left(  n\right)  }^{\prime}\right)  }.
\label{esn}%
\end{equation}
We can always choose the same order for the commutation factor
\textquotedblleft vector\textquotedblright\ components $\sigma$. Define a new
factor set $\left\{  \pi_{12}^{\left(  n\right)  }\right\}  $ by%
\begin{equation}
\pi_{12}^{\left(  n\right)  }\left(  a_{1}^{\prime},\ldots,a_{n}^{\prime
}\right)  =\frac{\pi_{1}^{\left(  n\right)  }\left(  a_{1}^{\prime}%
,\ldots,a_{n}^{\prime}\right)  }{\pi_{2}^{\left(  n\right)  }\left(
a_{1}^{\prime},\ldots,a_{n}^{\prime}\right)  },\ \ \ a_{1}^{\prime}%
,\ldots,a_{n}^{\prime}\in G. \label{p12}%
\end{equation}

Then $\left\{  \pi_{12}^{\left(  n\right)  }\right\}  $ becomes totally
symmetric, because the corresponding communication factor is%
\begin{equation}
\varepsilon_{12,\sigma\left(  1\right)  ,\ldots,\sigma\left(  n\right)
}^{\left(  n\right)  }\left(  a_{1}^{\prime},\ldots,a_{n}^{\prime}\right)
=\frac{\pi_{12}^{\left(  n\right)  }\left(  a_{1}^{\prime},\ldots
,a_{n}^{\prime}\right)  }{\pi_{12}^{\left(  n\right)  }\left(  a_{\sigma
\left(  1\right)  }^{\prime},\ldots,a_{\sigma\left(  n\right)  }^{\prime
}\right)  }=1,
\end{equation}
as follows from (\ref{esn})--(\ref{p12}). Therefore, as in the binary case, if
the grading group $\mathcal{G}$ is finitely generated and $\Bbbk$ is
algebraically closed, the communication factor $\varepsilon^{\left(  n\right)
}$ is constructed from the unique multiplier $\left\{  \pi^{\left(  n\right)
}\right\}  $ \cite{sch79}.

The are two possible differences from the simple binary case (\ref{esym}),
where for identity commutation factor the symmetry condition (\ref{psym}) is
sufficient: 1) Not all $\varepsilon^{\left(  n\right)  }$ need be equal to the
identity. 2) Some arguments of the Schur-like factors $\pi^{\left(  n\right)
}$ can be intact.

\begin{definition}
An $\varepsilon^{\left(  n\right)  }$-commutative $n$-ary graded algebra
$\mathcal{A}_{\mathcal{G}}^{\left(  n\right)  }$ is called $m$%
\textit{-partially }(or \textit{partially})\textit{ commutative} if exactly
$m$ commutation factors (with permutations $\tilde{\sigma}$) from the $\left(
n!-1\right)  $ total are equal to $1$%
\begin{equation}
\varepsilon_{\tilde{\sigma}\left(  1\right)  ,\ldots,\tilde{\sigma}\left(
n\right)  }^{\left(  n\right)  }\left(  a_{1}^{\prime},\ldots,a_{n}^{\prime
}\right)  =1,\ \ \#\tilde{\sigma}\leq\#\sigma,\sigma\in S_{n},\ \ a_{1}%
^{\prime},\ldots,a_{n}^{\prime}\in G, \label{part}%
\end{equation}
where we denote $\#\tilde{\sigma}=m$, $\#\sigma=n!-1$. If $\#\tilde{\sigma
}=\#\sigma$, then $\mathcal{A}_{\mathcal{G}}^{\left(  n\right)  }$ is totally
commutative, see (\ref{tsym}).
\end{definition}

In an $m$-partially commutative $n$-ary graded algebra $\mathcal{A}%
_{\mathcal{G}}^{\left(  n\right)  }$ the Schur-like factors $\pi^{\left(
n\right)  }$ satisfy $m$ additional symmetry conditions (cf. (\ref{psym}))%
\begin{equation}
\pi^{\left(  n\right)  }\left(  a_{1}^{\prime},\ldots,a_{n}^{\prime}\right)
=\pi^{\left(  n\right)  }\left(  a_{\tilde{\sigma}\left(  1\right)  }^{\prime
},\ldots,a_{\tilde{\sigma}\left(  n\right)  }^{\prime}\right)  ,\ \ \#\tilde
{\sigma}\leq\#\sigma,\ \ \sigma\in S_{n}.
\end{equation}

\begin{example}
[\textsf{Ternary }$\varepsilon$\textsf{-commutative algebra}]Let
$\mathcal{A}_{\mathcal{G}}^{\left(  3\right)  }=\left\langle A\mid\left[
\ \right]  ,+\right\rangle $ be a ternary associative $G$-graded algebra over
$\Bbbk$. The ternary Schur-like factor $\pi_{0}\left(  a_{1}^{\prime}%
,a_{2}^{\prime},a_{3}^{\prime}\right)  $ satisfies the ternary 2-cocycle
conditions (cf. (\ref{px0}))%
\begin{align}
&  \pi^{\left(  3\right)  }\left(  a_{1}^{\prime},a_{2}^{\prime},a_{3}%
^{\prime}\right)  \pi^{\left(  3\right)  }\left(  a_{1}^{\prime}+a_{2}%
^{\prime}+a_{3}^{\prime},a_{4}^{\prime},a_{5}^{\prime}\right) \nonumber\\
&  =\pi^{\left(  3\right)  }\left(  a_{2}^{\prime},a_{3}^{\prime}%
,a_{4}^{\prime}\right)  \pi^{\left(  3\right)  }\left(  a_{1}^{\prime}%
,a_{2}^{\prime}+a_{3}^{\prime}+a_{4}^{\prime},a_{5}^{\prime}\right)
\nonumber\\
&  =\pi^{\left(  3\right)  }\left(  a_{3}^{\prime},a_{4}^{\prime}%
,a_{5}^{\prime}\right)  \pi^{\left(  3\right)  }\left(  a_{1}^{\prime}%
,a_{2}^{\prime},a_{3}^{\prime}+a_{4}^{\prime}+a_{5}^{\prime}\right)
,\ \ a_{1}^{\prime},\ldots,a_{5}^{\prime}\in G.
\end{align}

We can introduce $\left(  3!-1\right)  =5$ different ternary commutation
relations, explicitly%
\begin{align}
&  \left[  a_{1},a_{2},a_{3}\right]  =\varepsilon_{\sigma\left(  1\right)
\sigma\left(  2\right)  \sigma\left(  3\right)  }^{\left(  3\right)  }\left[
a_{\sigma\left(  1\right)  },a_{\sigma\left(  2\right)  },a_{\sigma\left(
3\right)  }\right]  ,\ \ \ a_{1},a_{2},a_{3}\in A,,\ \sigma\in S_{3}%
,\label{a3}\\
\varepsilon_{132}^{\left(  3\right)  }  &  =\frac{\pi^{\left(  3\right)
}\left(  a_{1}^{\prime},a_{2}^{\prime},a_{3}^{\prime}\right)  }{\pi^{\left(
3\right)  }\left(  a_{1}^{\prime},a_{3}^{\prime},a_{2}^{\prime}\right)
},\ \ \varepsilon_{231}^{\left(  3\right)  }=\frac{\pi^{\left(  3\right)
}\left(  a_{1}^{\prime},a_{2}^{\prime},a_{3}^{\prime}\right)  }{\pi^{\left(
3\right)  }\left(  a_{2}^{\prime},a_{3}^{\prime},a_{1}^{\prime}\right)
},\ \ \varepsilon_{213}^{\left(  3\right)  }=\frac{\pi^{\left(  3\right)
}\left(  a_{1}^{\prime},a_{2}^{\prime},a_{3}^{\prime}\right)  }{\pi^{\left(
3\right)  }\left(  a_{2}^{\prime},a_{1}^{\prime},a_{3}^{\prime}\right)
},\nonumber\\
\varepsilon_{312}^{\left(  3\right)  }  &  =\frac{\pi^{\left(  3\right)
}\left(  a_{1}^{\prime},a_{2}^{\prime},a_{3}^{\prime}\right)  }{\pi^{\left(
3\right)  }\left(  a_{3}^{\prime},a_{1}^{\prime},a_{2}^{\prime}\right)
},\ \ \varepsilon_{321}^{\left(  3\right)  }=\frac{\pi^{\left(  3\right)
}\left(  a_{1}^{\prime},a_{2}^{\prime},a_{3}^{\prime}\right)  }{\pi^{\left(
3\right)  }\left(  a_{3}^{\prime},a_{2}^{\prime},a_{1}^{\prime}\right)
},\ \ a_{1}^{\prime},a_{2}^{\prime},a_{3}^{\prime}\in G.
\end{align}

The $1$-partial commutativity in $\mathcal{A}_{\mathcal{G}}^{\left(  3\right)
}$ can be realized, if, for instance,
\begin{equation}
\pi^{\left(  3\right)  }\left(  a_{1}^{\prime},a_{2}^{\prime},a_{3}^{\prime
}\right)  =\pi^{\left(  3\right)  }\left(  a_{1}^{\prime},a_{3}^{\prime}%
,a_{2}^{\prime}\right)  ,\ \ a_{1}^{\prime},a_{2}^{\prime},a_{3}^{\prime}\in
G,
\end{equation}
and in this case $\varepsilon_{132}^{\left(  3\right)  }=1$ so that we obtain
one commutativity relation%
\begin{equation}
\left[  a_{1},a_{2},a_{3}\right]  =\left[  a_{1},a_{3},a_{2}\right]
,\ \ \ a_{1},a_{2},a_{3}\in A,
\end{equation}
while the other $4$ relations in (\ref{a3}) will be $\varepsilon$-commutative.
\end{example}

\subsection{Membership deformed $n$-ary algebras}

Now we consider $n$-ary algebras over obscure sets as their underlying sets,
where each element of them is endowed with the membership function $\mu$ as a
degree of truth\ (see \textsc{Section} \ref{sec-memb}).

\begin{definition}
\label{def-n-balg}An \textit{obscure }$n$-\textit{ary algebra }$\mathcal{A}%
^{\left(  n\right)  }\left(  \mu\right)  =\left\langle \mathfrak{A}^{\left(
\mu\right)  }\mid\left[  \ \right]  ^{\left(  n\right)  },+\right\rangle $ is
an $n$-ary algebra $\mathcal{A}^{\left(  n\right)  }$ over $\Bbbk$ having an
obscure set $\mathfrak{A}^{\left(  \mu\right)  }=\left\{  \left(  a|\mu\left(
a\right)  \right)  ,a\in A,\mu>0\right\}  $ as its underlying set (see
(\ref{am})), and where the membership function $\mu\left(  a\right)  $
satisfies%
\begin{align}
\mu\left(  a_{1}+a_{2}\right)   &  \geq\min\left\{  \mu\left(  a_{1}\right)
,\mu\left(  a_{2}\right)  \right\}  ,\\
\mu\left(  \left[  a_{1},\ldots,a_{n}\right]  _{n}\right)   &  \geq
\min\left\{  \mu\left(  a_{1}\right)  ,\ldots,\mu\left(  a_{n}\right)
\right\}  ,\\
\mu\left(  ka\right)   &  \geq\mu\left(  a\right)  ,\ \ \forall a,a_{i}\in
A,\ \ k\in\Bbbk,\ \ i=1,\ldots,n.
\end{align}

\end{definition}

\begin{definition}
An\textit{ obscure} $G$-\textit{graded }$n$-ary \textit{algebra }%
$\mathcal{A}_{\mathcal{G}}^{\left(  n\right)  }\left(  \mu\right)  $ is a
direct sum decomposition (cf. (\ref{uu}))%
\begin{equation}
\mathcal{A}_{\mathcal{G}}^{\left(  n\right)  }\left(  \mu\right)
=\mathbf{\bigoplus}_{g\in G}\mathcal{A}_{g}^{\left(  n\right)  }\left(
\mu_{g}\right)  ,
\end{equation}
where $\mathfrak{A}^{\left(  \mu\right)  }=%
{\textstyle\bigcup}
{}_{g\in G}\mathfrak{A}_{g}^{\left(  \mu_{g}\right)  }$, $\mathfrak{A}%
^{\left(  \mu_{g}\right)  }=\left\{  \left(  a|\mu_{g}\left(  a\right)
\right)  ,a\in A_{g},\mu_{g}=(0,1]\right\}  $, and the joint membership
function $\mu$ is given by (\ref{ma}).
\end{definition}

In the obscure totally commutative $n$-ary algebra $\mathcal{A}_{\mathcal{G}%
}^{\left(  n\right)  }\left(  \mu\right)  $ (for homogeneous elements) we have
$\left(  n!-1\right)  $ commutativity relations (cf. (\ref{fcom}))%
\begin{equation}
\left[  a_{1},\ldots,a_{n}\right]  ^{\left(  n\right)  }=\left[
a_{\sigma\left(  1\right)  },\ldots,a_{\sigma\left(  n\right)  }\right]
^{\left(  n\right)  },\forall a_{1},\ldots,a_{n}\in\mathfrak{A}^{\left(
\mu\right)  },\ \forall\sigma\in S_{n},\ \sigma\neq I. \label{aa}%
\end{equation}

Let us consider a \textquotedblleft linear\textquotedblright\ (in $\mu$)
deformation of (\ref{aa}) analogous to the binary case (\ref{mab}).

\begin{definition}
An \textit{obscure membership deformed }$n$\textit{-ary algebra} is
$\mathcal{A}_{\ast\mathcal{G}}^{\left(  n\right)  }\left(  \mu\right)
=\left\langle \mathfrak{A}^{\left(  \mu\right)  }\mid\left[  \ \ \right]
_{\ast}^{\left(  n\right)  },+\right\rangle $ in which there are $\left(
n!-1\right)  $ possible noncommutativity relations%
\begin{align}
\mu_{a_{n}^{\prime}}\left(  a_{n}\right)  \left[  a_{1},\ldots,a_{n}\right]
_{\ast}^{\left(  n\right)  } &  =\mu_{a_{\sigma\left(  n\right)  }^{\prime}%
}\left(  a_{\sigma\left(  n\right)  }\right)  \left[  a_{\sigma\left(
1\right)  },\ldots,a_{\sigma\left(  n\right)  }\right]  _{\ast}^{\left(
n\right)  },\label{man}\\
\forall a_{1},\ldots,a_{n} &  \in\mathfrak{A}^{\left(  \mu\right)  }%
,\ \ a_{n}^{\prime},a_{\sigma\left(  n\right)  }^{\prime}\in G,\ \ \forall
\sigma\in S_{n},\ \sigma\neq I.\nonumber
\end{align}

\end{definition}

Since $\mu>0$, we can have

\begin{definition}
The $\left(  n!-1\right)  $ membership commutation factors in $\mathcal{A}%
_{\ast\mathcal{G}}^{\left(  n\right)  }\left(  \mu\right)  $ are defined by%
\begin{equation}
\mathbf{\epsilon}_{\sigma\left(  n\right)  }^{\left(  n\right)  }\left(
a_{n}^{\prime},a_{\sigma\left(  n\right)  }^{\prime},a_{n},a_{\sigma\left(
n\right)  }\right)  =\frac{\mu_{a_{\sigma\left(  n\right)  }^{\prime}}\left(
a_{\sigma\left(  n\right)  }\right)  }{\mu_{a_{n}^{\prime}}\left(
a_{n}\right)  },\ \ a_{n}^{\prime},a_{\sigma\left(  n\right)  }^{\prime}\in
G,\ \ \forall\sigma\in S_{n},\ \sigma\neq I, \label{ena}%
\end{equation}
and the relations (\ref{man}) become
\begin{equation}
\left[  a_{1},\ldots,a_{n}\right]  _{\ast}^{\left(  n\right)  }%
=\mathbf{\epsilon}_{\sigma\left(  n\right)  }^{\left(  n\right)  }\left(
a_{n}^{\prime},a_{\sigma\left(  n\right)  }^{\prime},a_{n},a_{\sigma\left(
n\right)  }\right)  \left[  a_{\sigma\left(  1\right)  },\ldots,a_{\sigma
\left(  n\right)  }\right]  _{\ast}^{\left(  n\right)  }.
\end{equation}

\end{definition}

These definitions are unique, if we require: 1) for connection only two
monomials with different permutations; 2) membership \textquotedblleft
linearity\textquotedblright\ (in $\mu$); 3) compatibility with the binary case
(\ref{abe})--(\ref{emm}).

\begin{example}
For the obscure membership deformed ternary algebra $\mathcal{A}%
_{\ast\mathcal{G}}^{\left(  3\right)  }\left(  \mu\right)  =\left\langle
\mathfrak{A}^{\left(  \mu\right)  }\mid\left[  \ \ \right]  _{\ast}^{\left(
3\right)  },+\right\rangle $, we obtain%
\begin{align}
&  \left[  a_{1},a_{2},a_{3}\right]  =\mathbf{\epsilon}_{\sigma\left(
1\right)  \sigma\left(  2\right)  \sigma\left(  3\right)  }^{\left(  3\right)
}\left[  a_{\sigma\left(  1\right)  },a_{\sigma\left(  2\right)  }%
,a_{\sigma\left(  3\right)  }\right]  _{\ast}^{\left(  3\right)  }%
,\ \ \ a_{1},a_{2},a_{3}\in A,,\ \sigma\in S_{3},\\
\mathbf{\epsilon}_{132}^{\left(  3\right)  }  &  =\mathbf{\epsilon}%
_{312}^{\left(  3\right)  }=\frac{\mu_{a_{3}^{\prime}}\left(  a_{3}\right)
}{\mu_{a_{2}^{\prime}}\left(  a_{2}\right)  },\ \ \mathbf{\epsilon}%
_{231}^{\left(  3\right)  }=\mathbf{\epsilon}_{321}^{\left(  3\right)  }%
=\frac{\mu_{a_{3}^{\prime}}\left(  a_{3}\right)  }{\mu_{a_{1}^{\prime}}\left(
a_{1}\right)  },\ \ \mathbf{\epsilon}_{213}^{\left(  3\right)  }%
=1,,\ \ \ a_{1}^{\prime},a_{2}^{\prime},a_{3}^{\prime}\in G,\nonumber
\end{align}
which means that it is $1$-partially commutative (see (\ref{part})) and has
only $2$ independent membership commutation factors (cf. the binary case
(\ref{emm})).
\end{example}

We now provide a sketch construction of the $n$-ary $\varepsilon$-commutative
algebras (\ref{an}) membership deformation (see \textbf{Subsection}
\ref{subsec-def-e} for the binary case).

\begin{definition}
An \textit{obscure membership deformed }$n$\textit{-ary }$\pi$%
\textit{-commutative algebra} over $\Bbbk$ is $\mathcal{A}_{\star\mathcal{G}%
}^{\left(  n\right)  }\left(  \mu\right)  =\left\langle \mathfrak{A}^{\left(
\mu\right)  }\mid\left[  \ \ \right]  _{\star}^{\left(  n\right)
},+\right\rangle $ in which the following $\left(  n!-1\right)  $
noncommutativity relations are valid%
\begin{align}
&  \pi^{\left(  n\right)  }\left(  a_{\sigma\left(  1\right)  }^{\prime
},\ldots,a_{\sigma\left(  n\right)  }^{\prime}\right)  \mu_{a_{n}^{\prime}%
}\left(  a_{n}\right)  \left[  a_{1},\ldots,a_{n}\right]  _{\star}^{\left(
n\right)  }\nonumber\\
&  =\pi^{\left(  n\right)  }\left(  a_{1}^{\prime},\ldots,a_{n}^{\prime
}\right)  \mu_{a_{\sigma\left(  n\right)  }^{\prime}}\left(  a_{\sigma\left(
n\right)  }\right)  \left[  a_{\sigma\left(  1\right)  },\ldots,a_{\sigma
\left(  n\right)  }\right]  _{\star}^{\left(  n\right)  },\forall a_{i}%
\in\mathfrak{A}^{\left(  \mu\right)  },\ \ a_{i}^{\prime}\in G,\ \ \forall
\sigma\in S_{n},\ \sigma\neq I,
\end{align}
where $\pi^{\left(  n\right)  }$ are $n$-ary Schur-like factors satisfying the
2-cocycle conditions (\ref{pa}).
\end{definition}

\begin{definition}
An algebra $\mathcal{A}_{\star\mathcal{G}}^{\left(  n\right)  }\left(
\mu\right)  $ is called a \textit{double} $\varepsilon_{\pi}^{\left(
n\right)  }/\mathbf{\epsilon}_{\mu}^{\left(  n\right)  }$-\textit{commutative
algebra}, if%
\begin{align}
&  \left[  a_{1},\ldots,a_{n}\right]  _{\star}^{\left(  n\right)  }
=\varepsilon_{\sigma\left(  1\right)  ,\ldots,\sigma\left(  n\right)
}^{\left(  n\right)  }\left(  a_{1}^{\prime},\ldots,a_{n}^{\prime}\right)
\mathbf{\epsilon}_{\sigma\left(  n\right)  }^{\left(  n\right)  }\left(
a_{n}^{\prime},a_{\sigma\left(  n\right)  }^{\prime},a_{n},a_{\sigma\left(
n\right)  }\right)  \left[  a_{\sigma\left(  1\right)  },\ldots,a_{\sigma
\left(  n\right)  }\right]  _{\star}^{\left(  n\right)  },\\
&  \varepsilon_{\sigma\left(  1\right)  ,\ldots,\sigma\left(  n\right)
}^{\left(  n\right)  }\left(  a_{1}^{\prime},\ldots,a_{n}^{\prime}\right)
=\frac{\pi^{\left(  n\right)  }\left(  a_{1}^{\prime},\ldots,a_{n}^{\prime
}\right)  }{\pi^{\left(  n\right)  }\left(  a_{\sigma\left(  1\right)
}^{\prime},\ldots,a_{\sigma\left(  n\right)  }^{\prime}\right)  },\\
&  \mathbf{\epsilon}_{\sigma\left(  n\right)  }^{\left(  n\right)  }\left(
a_{n}^{\prime},a_{\sigma\left(  n\right)  }^{\prime},a_{n},a_{\sigma\left(
n\right)  }\right)  =\frac{\mu_{a_{\sigma\left(  n\right)  }^{\prime}}\left(
a_{\sigma\left(  n\right)  }\right)  }{\mu_{a_{n}^{\prime}}\left(
a_{n}\right)  },\\
&  \ a_{1},\ldots,a_{n},a_{\sigma\left(  1\right)  },\ldots,a_{\sigma\left(
n\right)  } \in\mathfrak{A}^{\left(  \mu\right)  },\ \ a_{1}^{\prime}%
,\ldots,a_{n}^{\prime},a_{\sigma\left(  1\right)  }^{\prime},\ldots
,a_{\sigma\left(  n\right)  }^{\prime}\in G,\ \ \forall\sigma\in
S_{n},\ \sigma\neq I,
\end{align}
where $\varepsilon_{\pi}^{\left(  n\right)  }$ is the $n$-ary grading
commutation factor (\ref{en}) and $\mathbf{\epsilon}_{\mu}^{\left(  n\right)
}$ is the $n$-ary membership commutation factor (in our definition (\ref{ena})).
\end{definition}

This procedure can be considered to be the membership deformation of the given
$\varepsilon$-commutative algebra (\ref{an}), which corresponds to the first
version of the binary commutative algebra deformation as in (\ref{mp2}) and
leads to a nonassociative algebra, in general. To achieve associativity one
should consider the second version of the algebra deformation as in the binary
case (\ref{ab1}): to introduce a different $n$-ary multiplication $\left[
\ \ \right]  _{\circledast}^{\left(  n\right)  }$ such that the product of
commutation factors $\varepsilon^{\left(  n\right)  }\mathbf{\epsilon
}^{\left(  n\right)  }$ satisfies the $n$-ary $2$-cocycle-like\ conditions,
while the $n$-ary noncocycle commutation factor satisfies the
\textquotedblleft membership deformed cocycle-like\textquotedblright%
\ conditions, analogous to (\ref{em1})--(\ref{em3}).

\bigskip

\textbf{Acknowledgements.} \addcontentsline{toc}{section}{Acknowledgements}
The author wishes to express his sincere thankfulness and deep gratitude to
Andrew James Bruce, Grigorij Kurinnoj, Mike Hewitt, Thomas Nordahl, Vladimir
Tkach, Raimund Vogl, Alexander Voronov, and Wend Werner for numerous fruitful
discussions and valuable support.

\newpage
\pagestyle{emptyf}
\mbox{}

\end{document}